\documentclass[11pt]{amsart}
\usepackage{eucal,fullpage,times,amsmath,amsthm,amssymb,mathrsfs,stmaryrd}
\usepackage[all]{xy}
\usepackage{url}

\newcommand{\calN}{{\mathcal{N}}}
\newcommand{\calX}{{\mathcal{X}}}
\newcommand{\calO}{{\mathcal{O}}}
\newcommand{\calL}{{\mathcal{L}}}
\newcommand{\calU}{{\mathcal{U}}}
\newcommand{\calV}{{\mathcal{V}}}
\newcommand{\calA}{{\mathcal{A}}}
\newcommand{\calM}{{\mathcal{M}}}
\newcommand{\calD}{{\mathcal{D}}}
\newcommand{\calQ}{\mathcal{Q}}
\newcommand{\calE}{\mathcal{E}}
\newcommand{\calF}{\mathcal{F}}

\newcommand{\calH}{\mathcal{H}}
\newcommand{\calC}{\mathcal{C}}
\newcommand{\calHom}{\mathcal{H}\mathrm{om}}

\newcommand{\A}{\mathbf{A}}
\newcommand{\G}{\mathbf{G}}
\newcommand{\Z}{\mathbf{Z}}
\newcommand{\C}{\mathbf{C}}
\newcommand{\F}{\mathbf{F}}
\newcommand{\Q}{\mathbf{Q}}
\renewcommand{\P}{\mathbf{P}}

\newcommand{\Spec}{{\mathrm{Spec}}}

\newcommand{\Hom}{\mathrm{Hom}}

\newcommand{\GL}{\mathrm{GL}}

\newcommand{\Coh}{\mathrm{Coh}}

\newcommand{\Vect}{\mathrm{Vect}}
\newcommand{\Sym}{\mathrm{Sym}}
\newcommand{\Frob}{\mathrm{Frob}}

\newcommand{\Ab}{\mathrm{Ab}}
\newcommand{\Gr}{\mathrm{Gr}}
\renewcommand{\ss}{\mathrm{ss}}
\newcommand{\ch}{\mathrm{char}}
\newcommand{\D}{\mathrm{D}}
\newcommand{\R}{\mathrm{R}}
\renewcommand{\L}{\mathrm{L}}
\newcommand{\W}{\mathrm{W}}
\newcommand{\Pic}{\mathrm{Pic}}
\newcommand{\et}{\mathrm{\acute{e}t}}

\newcommand{\im}{\mathrm{im}}
\newcommand{\id}{\mathrm{id}}
\newcommand{\ev}{\mathrm{ev}}

\newcommand{\Fil}{\mathrm{Fil}}
\newcommand{\pr}{\mathrm{pr}}
\newcommand{\Tr}{\mathrm{Tr}}

\newcommand{\Frac}{\mathrm{Frac}}

\newcommand{\Flag}{\mathrm{Flag}}
\newcommand{\univ}{\mathrm{univ}}

\newcommand{\fm}{\mathfrak{m}}

\newcommand{\fp}{\mathfrak{p}}

\newcommand{\comment}[1]{}

\begin{document}

\bibliographystyle{alpha}

\newtheorem{theorem}{Theorem}[section]
\newtheorem*{theorem*}{Theorem}
\newtheorem*{condition*}{Condition}
\newtheorem*{definition*}{Definition}
\newtheorem{proposition}[theorem]{Proposition}
\newtheorem{lemma}[theorem]{Lemma}
\newtheorem{corollary}[theorem]{Corollary}
\newtheorem{claim}[theorem]{Claim}
\newtheorem{claimex}{Claim}[theorem]

\theoremstyle{definition}
\newtheorem{definition}[theorem]{Definition}
\newtheorem{question}[theorem]{Question}
\newtheorem{remark}[theorem]{Remark}
\newtheorem{example}[theorem]{Example}
\newtheorem{condition}[theorem]{Condition}
\newtheorem{warning}[theorem]{Warning}
\newtheorem{notation}[theorem]{Notation}

\title{Derived splinters in positive characteristic}
\author{Bhargav Bhatt}
\begin{abstract}
This paper introduces the notion of a derived splinter. Roughly speaking, a scheme is a derived splinter if it splits off from the coherent cohomology of any proper cover. Over a field of characteristic $0$, this condition characterises rational singularities by a result of Kov\'acs. Our main theorem asserts that over a field of characteristic $p$, derived splinters are the same as (underived) splinters, i.e., as schemes that split off from any finite cover. Using this result, we answer some questions of Karen Smith concerning extending Serre/Kodaira type vanishing results beyond the class of ample line bundles in positive characteristic; these are purely projective geometric statements independent of singularity considerations. In fact, we can prove ``up to finite cover'' analogues in characteristic $p$ of many vanishing theorems one knows in characteristic $0$. All these results fit naturally in the study of F-singularities, and are motivated by a desire to understand the direct summand conjecture.
\end{abstract}

\maketitle

The goal of this paper is to introduce the notion of derived splinters and prove some basic results about them. Since derived splinters are analogues of splinters, we review the definition of the latter first.

\begin{definition}
A scheme $S$ is called a {\em splinter} if for any finite surjective map $f:X \to S$, the pullback map $\calO_S \to f_* \calO_X$ is split in the category $\Coh(S)$ of coherent sheaves on $S$.
\end{definition}

This term was coined in \cite{SinghQGorSplinter}, though the idea is much older. In characteristic $0$, splinters are exactly normal schemes (see Example \ref{ex:char0}). Away from characteristic $0$, however, splinters become much more interesting, and are the subject of numerous results and questions in commutative algebra. In particular, the still open direct summand conjecture \cite{HochsterDSC} posits that all regular rings are splinters; this conjecture is known in equicharacteristic by \cite{HochsterDSC} or in mixed characteristic for dimensions $\leq 3$ by \cite{HeitmannDSC3}, but is unknown in general, and is a fundamental open problem in the subject.

Our definition of derived splinters is inspired by that of splinters and the philosophy that proper maps provide robust derived analogues of finite maps, at least for coherent sheaf theory. More precisely, we have:

\begin{definition}
A scheme $S$ is called a {\em derived splinter}, or simply a {\em D-splinter}, if for any proper surjective map $f:X \to S$, the pullback map $\calO_S \to \R f_* \calO_X$ is split in derived category $D(\Coh(S))$ of coherent sheaves on $S$.
\end{definition}

Like splinters, the idea of D-splinters is not new. In fact, a theorem of Kov\'acs \cite{KovacsRational} identifies D-splinters in characteristic $0$ with schemes whose singularities are at worst rational. Since splinters in characteristic $0$ are precisely normal schemes, one notices immediately that splinters and D-splinters define extremely different classes of singularities in characteristic $0$. In characteristic $p > 0$ however, we discover a remarkably different picture; one of the main theorems of this paper is the following:

\begin{theorem}
\label{ddsct}
A noetherian $\F_p$-scheme $S$ is a splinter if and only if it is a D-splinter.
\end{theorem}

The main tool used to prove Theorem \ref{ddsct} is a cohomology-annihilation result that is of independent interest: we show that the higher cohomology of the structure sheaf on a projective variety in characteristic $p$ can always be killed by a finite cover. In fact, we prove the following stronger relative statement:

\begin{theorem}
\label{killdercoh}
Let $f:X \to S$ be a proper morphism of noetherian $\F_p$-schemes. Then there exists a proper morphism $g:Y \to S$ and a finite surjective morphism $\pi:Y \to X$ such that the pullback map $\pi^*:\tau_{\geq 1} \R f_*\calO_X \to \tau_{\geq 1} \R g_*\calO_Y$ is $0$.
\end{theorem}

The proof of Theorem \ref{killdercoh} is inspired by the paper \cite{HHBigCM} of M. Hochster and C. Huneke proving the existence of big Cohen-Macaulay algebras in positive characteristic (and also \cite{GLBigCM}). The same paper led K. Smith to ask certain questions concerning extensions of the vanishing theorems of Serre and Kodaira beyond the ample cone (see \S \ref{sec:poscharappksmith}). Using our methods, we are able to  answer these questions. The negative answers are recorded in the form of counterexamples at the end of \S \ref{sec:poscharappksmith}, while the affirmative answers are summarised below in Theorem \ref{smithconjsumm}. We refer the reader to Propositions \ref{smithconjpos} and \ref{smithconjneg} for more precise statements.

\begin{theorem}
\label{smithconjsumm}
Let $X$ be a proper variety over a field $k$ of positive characteristic, and let $\calL$ be a semiample line bundle on $X$. Then $H^i(X,\calL)$ can be killed by finite covers for $i > 0$. If $\calL$ is big as well, then $H^i(X,\calL^{-1})$ can be killed by finite covers for $i < \dim(X)$. 
\end{theorem}

It seems worthwhile to remark at this point that the results mentioned above, in conjunction with those proven in \S \ref{sec:poscharappksmith}, have applications unrelated to splinters or D-splinters: these results suggest that numerous vanishing theorems that are true in characteristic $0$ have analogues in characteristic $p$ provided one works ``up to finite covers.'' This idea has been pursued in much more depth in the recent work \cite{BSTFSingAlt}, where ``up to finite cover'' analogues of the Nadel vanishing have been established.

Returning to affine D-splinters, we note that in positive characteristic $p$, by Theorem \ref{ddsct}, this class of singularities is closely related to other classes of singularities, the so-called F-singularities, defined using the Frobenius action. For example, locally excellent affine $\Q$-Gorenstein splinters are F-regular by \cite{SinghQGorSplinter} which builds on the Gorenstein case proven in \cite{HHTCSplinter}; see also Example \ref{ex:fratdsc} below.  In contrast, in the case of mixed characteristic, our knowledge about either splinters or D-splinters is minimal, primarily because the direct summand conjecture is unknown. For progress towards establishing an analogue of Theorem \ref{ddsct}, and especially a weak mixed characteristic analogue of Theorem \ref{killdercoh}, we refer the reader to \cite{Bhattmixedcharpdiv}.

\subsubsection*{Organisation of this paper}  The purpose of \S \ref{sec:ex} is to collect some examples and non-examples of splinters and D-splinters; the goal here is to describe some of the geometry underlying these definitions.  In \S \ref{sec:poschardercat} we review some general notation and results about derived categories used in this paper; the key result is a method for passing from conclusions at the level of cohomology groups to those at the level of complexes. Theorems \ref{ddsct} and \ref{killdercoh} are proven in \S \ref{sec:poscharmainthm}, and some refinements are proven in \S \ref{sec:poscharcommentary}; the method here is inspired by that of \cite{GLBigCM} and, by transitivity, by that of \cite{HHBigCM}. Moving to applications, we discuss some purely algebraic applications of the preceding theorems in \S \ref{sec:poscharappalg}. In \S \ref{sec:poscharappksmith}, we review some questions raised by Karen Smith in the wake of \cite{HHBigCM}, and then discuss both positive and negative answers we can provide; the highlights here are the ``up to finite covers'' version of Kodaira vanishing in Proposition \ref{smithconjneg}, and some of the counterexamples, especially Example \ref{nefnotenough}. Finally, in \S \ref{sec:poscharglobalex} we use Theorem \ref{killdercoh} to show that the complete flag variety for $\GL_n$ is a D-splinter, thereby providing the first non-toric projective example of one.

\subsubsection*{Acknowledgements}
This paper forms a part of the author's doctoral dissertation written under Aise Johan de Jong, and would not have been possible without his consistent support. In particular, the author would like to thank de Jong for suggesting some of the questions addressed in this paper, and for generously sharing ideas related to many parts of this paper. In addition, the author would also like to thank Karl Schwede and, especially, Anurag Singh for many conversations about derived splinters.

\section{Some examples}
\label{sec:ex}

The goal of this section is collect some examples of splinters and D-splinters. Since the notions in characteristic $0$ are quite well understood, we focus mainly on the case of characteristic $p$. Moreover, it is typically non-trivial to prove that any given ring is a splinter or D-splinter. Hence, we freely use results in the literature or elsewhere in this paper in our proofs; we hope that despite the resulting non-elementary nature of the examples, the reader will be convinced that splinters and D-splinters are geometrically interesting. 

We dispose of the characteristic $0$ case.

\begin{example}[Splinters in characteristic $0$]
\label{ex:char0}
A connected noetherian $\Q$-scheme $S$ is a splinter if and only if it is normal. For the forward direction, note that the map from the disjoint union of the irreducible components of $S$ to $S$ immediately shows that $S$ is forced to be a domain if it is a splinter. The desired claim now follows from the following ring-theoretic fact: if $R$ is an integral domain with $a/b \in \Frac(R)$ integral over $R$, then $R \to R[a/b]$ is not split unless $a/b \in R$. To prove this, we simply observe that if $R \to R[a/b]$ were split, then the quotient would be a torsion free $R$-module with generic rank $0$, which can only happen when the quotient is trivial. 

For the converse implication, we need to show that if $f:X \to S$ is a finite surjective morphism, and $S$ is normal and connected, then $f^*:\calO_S \to f_*\calO_X$ has a section in $\Coh(S)$. After replacing $X$ with an irreducible component dominating $S$, we may assume that $X$ is integral.  Let $d$ denote the degree of the map induced by $f$ at the level of function fields. Then the map $\frac{1}{d}\Tr_{X/S}$ provides a canonical splitting for the map $f^*:\calO_S \to f_*\calO_X$ (here we use that the trace map on function fields preserves integrality).
\end{example}

\begin{example}[D-splinters in characteristic $0$]
\label{ex:ddscchar0}
Let $X$ be a variety over $\C$. Then $X$ is D-splinter if and only if $X$ has rational singularities, i.e., if $\R f_* \calO_Y \simeq \calO_X$ for some (equivalently, every) resolution of singularities $f:Y \to X$. We refer the reader to \cite{KovacsRational} for a proof.
\end{example}

A splinter in positive characteristic $p$ is subtler than its characteristic $0$ avatar, as being a splinter imposes some kind of positivity (both local and global) on the variety. In fact, in view of Theorem \ref{ddsct}, over $\F_p$, being a splinter is equivalent to being a D-splinter, a condition that is {\em a priori} much more restrictive. Nevertheless, large classes of examples of splinters (or, equivalently, D-splinters) over $\F_p$ do exist, and are catalogued below. The intuition informing most of these examples is that splinters should be analogous to rational singularities in characteristic $0$.

\begin{example}[Smooth affines are splinters]
\label{ex:smdsc}
All regular affine $\F_p$-schemes are splinters; this is a result of Mel Hochster (see \cite{HochsterDSC}), and we record a proof of Hochster's theorem below for the convenience of the reader. The proof given below is cohomological in nature, and different from Hochster's.

We first explain the idea informally. Let $f:\Spec(S) \to \Spec(R)$ be a finite surjective map. Using the fact that $R$ is Gorenstein, an elementary duality argument will reduce us to showing that $H^d_\fm(R) \to H^d_\fm(S)$ is injective. The kernel of this map is a Frobenius stable proper submodule of $H^d_\fm(R)$ of finite length by an inductive argument due to Grothendieck (see \cite[Expos\'e VIII, Th\'eor\`eme 2.1]{SGA2}). The regularity of $R$ is regular will then imply that this is impossible for length reasons.

Now for the details. After localising and completing, we may assume that $(R,\fm)$ is a complete regular local $\F_p$-algebra of dimension $d$. By the Cohen structure theorem (see \cite[\S 28, Theorem 28.J and Corollary 2]{MatCA}), we know that $R \simeq k\llbracket x_1,\cdots,x_n \rrbracket$. Since field extensions $k \to L$ split as $k$-modules, we may pass to the algebraic closure of the coefficient field to assume that $k$ is algebraically closed. In particular, the Frobenius map $F:R \to R$ is finite. Given a finite extension $f:R \to S$, we need to show that the natural map $\ev_f:\Hom(S,R) \to \Hom(R,R)$ is surjective. By induction, we may assume that the cokernel $\calQ$ is supported only at the closed point $\{\fm\} \subset \Spec(R)$. In particular, the cokernel $\calQ$ has finite length. Now the fact that $R$ is Gorenstein implies that $\omega_R \simeq R$. Thus, the map $\ev_f$ is isomorphic to the trace map $\Hom(S,\omega_R) \to \omega_R$, which is dual to the canonical pullback map $H^d_\fm(f):H^d_\fm(R) \to H^d_\fm(S)$. As local duality interchanges kernels and cokernels while preserving lengths, it follows that the kernel $M := \ker(H^d_\fm(f))$ has the same length as $\calQ$; this kernel is also Frobenius-stable by construction. Now consider the diagram
\[ \xymatrix{
F^* M  \ar[d]^a \ar[r]^b & M \ar[d]^c \\
F^* H^d_\fm(R) \ar[r]^d & H^d_\fm(R). } \]
The map $a$ is injective since $F^*$ is exact (by regularity of $R$), while the map $d$ is an isomorphism by the flat base change isomorphism $\R \Gamma_\fm(R) \otimes_{R,F} R \simeq \R \Gamma_\fm(R)$ (see \cite[\S 4.3.2]{BSLocalCoh}). The diagram then shows that $b$ is also injective, and thus the length of $F^* M$ is bounded above by that of $M$. The claim now follows from the elementary observation that $F^*$ multiplies length by $d > 0$.
\end{example}

Following the proof of Example \ref{ex:smdsc} leads to a much larger class of splinters, defined in terms of F-rational rings. We remind the reader that a noetherian local $\F_p$-algebra $(R,\fm)$ of dimension $d$ is called {\em F-rational} if it is Cohen-Macaulay, normal, and has the property that $H^d_\fm(R)$ has no proper Frobenius-stable submodules except $0$. This is not the original definition of F-rationality, but equivalent to it by work of Karen Smith, see \cite{KStcparamideals, KSfratratsing}.

\begin{example}[F-rational Gorenstein rings are splinters]
\label{ex:fratdsc}
Let $(R,\fm)$ be a noetherian excellent local $\F_p$-algebra admitting a dualising complex. Assume that $R$ is Gorenstein. If $R$ is F-rational, then $R$ is a splinter. This follows from the proof given in Example \ref{ex:smdsc}. In more detail, we may assume without loss of generality that that $(R,\fm)$ is an F-rational Gorenstein complete noetherian local ring of dimension $d > 0$. Given a finite extension $f:R \to S$, we need to verify that $\ev_f:\Hom(S,R) \to R$ is surjective. By the Gorenstein assumption, we can identify this map with $\Hom(S,\omega_R) \to \omega_R$. The image of this last map is a Frobenius-stable submodule $M \subset \omega_R$. Moreover, since the formation of $M$ commutes with localisation, we know that $M$ is generically non-zero. The definition of F-rationality then shows that $M = \omega_R$ as desired.
\end{example}

\begin{remark}
One may show a converse to Example \ref{ex:fratdsc} as follows: any excellent splinter is $F$-rational. To see this, note that the argument in Example \ref{ex:char0} shows that $R$ is normal, while Corollary \ref{dsciscm} below shows that $R$ is Cohen-Macaulay. To show that $R$ is $F$-rational, one can then use \cite[Theorem 2.6]{KSfratratsing} and \cite[Theorem 5.4]{KStcparamideals}. Together, these theorems imply that it is enough to check that for all ideals $I$ generated by a system of parameters, we have $IS \cap R = I$ for all finite extensions $R \to S$. The splinter property implies that $IS = I \oplus Q$, which easily shows that $IS \cap R = I$.
\end{remark}

We work out a special case of Example \ref{ex:fratdsc}, to give an idea of the relevant geometry.

\begin{example}[The quadric cone]
\label{quadcone}
We claim that  $R = k\llbracket x_1,\dots,x_n \rrbracket /(\sum_i x_i^2)$ is a splinter for $n \geq 3$ provided $\ch(k) > 2$. By Example \ref{ex:fratdsc}, it suffices to show that $R$ is $F$-rational. By \cite[Theorem 4.2]{HunekeTCBook}, it suffices to show that $R/(x_n)$ is $F$-rational. Thus, we can set up an induction once we settle the $n = 3$ case. This case follows from \cite[Example 3]{HochsterDSC}. Alternately, in the $n = 3$ case, we may identify $R$ with (completion at the origin of) the affine cone on a smooth conic $C \subset \P^2$. Since $C$ is a hypersurface, the scheme $\Spec(R)$ has an isolated hypersurface singularity at $0$, and is thus Cohen-Macaulay and normal. Moreover, identifying $\Spec(R) - \{\fm\})$ with the total space of the complement of the $0$ section in $\calO_{\P^2}(-1)|_C$ shows that 
\[ H^2_\fm(R)  \simeq \oplus_{n \in \Z} H^1(C,\calO_{\P^2}(-n)|_C) \simeq \oplus_{n \in \Z} H^1(\P^1,\calO(2n)). \]
The preceding presentation is Frobenius equivariant, where Frobenius acts on the grading on the right by multiplying the weights by $p$. By inspection, it easily follows then that $H^2_\fm(R)$ has no Frobenius-stable proper non-zero submodules, proving F-rationality.
\end{example}

Next, we show that certain quotient singularities are splinters.

\begin{example}[Quotient singularities are often splinters]
\label{quotsing}
Let $k$ be a field, and let $R$ be a regular $k$-algebra. Let $G$ be a linearly reductive group acting on $R$. Then $\Spec(R^G)$ is a splinter. Indeed, the inclusion $R^G \to R$ has an $R^G$-linear section given by the Reynolds operator, and so the splinter property for $R^G$ follows from that of $R$. More generally, the same argument shows that any subring $A$ of a regular ring $R$ that splits off as an $A$-linear summand is a splinter. In particular, if $G$ is a reductive group over $\C$ acting on an affine algebraic $\C$-scheme $\Spec(R)$, then almost all positive characteristic reductions of $\Spec(R^G)$ are splinters.
\end{example}

We next list a large class of non-examples.

\begin{example}[General type cones are not splinters]
\label{ex:nonfanocone}
Let $X \subset \P^n$ be a hypersurface of degree $d > (n+1)$ over a perfect field $k$ of characteristic $p$, and let $S$ be the affine cone on $X$. Then $S$ is not a splinter. To see this, note that as in Example \ref{quadcone}, we have an identification
\[ H^n_\fm(S) \simeq \oplus_{i \in \Z} H^{n-1}(X,\calO_X(i)) \]
that is Frobenius equivariant, where Frobenius acts on the right by scaling the weights by $p$. Now $\omega_X \simeq \calO(d-n-1)|_X$ by adjunction. One then easily computes that $H^{n-1}(X,\omega^p) = H^{n-1}(X,\Frob_X^* \omega_X) = 0$, and thus $\Frob_X^*:H^{n-1}(X,\omega_X) \to H^{n-1}(X,\Frob_X^* \omega_X)$ has a non-trivial kernel. It follows that $\Frob_S^*:H^n_\fm(S) \to H^n_\fm(S)$ also has a non-trivial kernel, and so $\Frob_S:S \to S$ is not split.
\end{example}

Lastly, we discuss a non-example due to Hochster: a hypersurface singularity of dimension $2$ in characteristic $2$ that is not a splinter. Aside from its intrinsic interest, this example is meant to caution the reader as the standard lift of this hypersurface to characteristic $0$ has rational singularities.

\begin{example}
\label{hocgorendsc}
Let $k$ be a field of characteristic $2$. Let $S = k[u,v]$ be a polynomial ring, and let $R = k[u^2,v^2,u^3+v^3] \hookrightarrow S$. Since $\ch(k) = 2$, $R$ admits the presentation $R = k[x,y,z]/(x^3 + y^3 + z^2)$ where $x = u^2$, $y = v^2$, and $z = u^3 + v^3$. In particular, $\Spec(R)$ is a hypersurface singularity of dimension $2$. Since the singularity is isolated, $R$ is even normal. On the other hand, $\Spec(R)$ is not a splinter  because the natural map $f:\Spec(S) \to \Spec(R)$ is a finite surjective map such that $\calO_{\Spec(R)} \to f_*\calO_{\Spec(S)}$ has no section: identifying sheaves with modules and applying such a section $s$ to $u^3 + v^3 = u \cdot u^2 + v \cdot v^2$ would give us $u^3 + v^3 = s(u^3 + v^3) = s(u)u^2 + s(v)v^2 \in (u^2,v^2)R$ which is false. The same example can be adapted to arbitrary positive characteristic $p$ by setting $R = k[u^p,v^p,u^a + v^a]$ for some $p < a < 2p$.
\end{example}

The examples discussed hitherto were all affine. Requiring a projective variety $X$ over a positive characteristic field $k$ to be a splinter leads to questions of a very different flavour as the geometry of $X$ is heavily constrained. For example, Theorem \ref{killdercoh} shows that $H^i(X,\calO_X) = 0$ for all $i > 0$. In fact, the same theorem applied to a high iterate of Frobenius shows that $H^i(X,\calL) = 0$ for $i > 0$ whenever $\calL$ is an ample line bundle. Thus, projective examples are harder to find; nevertheless, they do exist, as we show below. We will discuss such examples further in \S \ref{sec:poscharglobalex}. 

\begin{example}[Toric varieties are often splinters]
\label{toricex}
Any toric variety $X$ that is projective over an affine is a splinter. To see this, note that any such $X$ can be obtained as a quotient $U/G$ (see \cite[Theorem 10.27]{MillerSturmfelsCCA}), where $U \subset \A^n$ is an open subscheme, and  $G \subset \G_m^n$ is an algebraic subgroup preserving $U$. As $\G_m^n$ is linearly reductive, so is $G$ (see \cite[Proposition 2.5]{AOVTameStacks}). In particular, we see that $\calO_X \to \pi_*\calO_U \simeq \R \pi_*\calO_U$ is a direct summand. The result now follows from the fact that $U$ is a splinter, which in turn follows from Example \ref{ex:smdsc} and the fact any finite cover of $U$ comes from a finite cover of $\A^n$ (by normalisation, for example).
\end{example}

Lastly, we record an elementary example showing that not all smooth projective varieties are splinters.

\begin{example}
\label{globaldscfailure}
Let $E$ be an elliptic curve over an algebraically closed field $k$ of positive characteristic $p$. We will show that $E$ is not a splinter. Consider the multiplication by $p$ map $[p]:E \to E$. The induced map on $H^1(E,\calO_E)$ can be easily seen to be $0$; for example, one can show that $[n]^*$ induces multiplication by $n$ on $H^1(A,\calO_A)$ for any abelian variety $A$. It follows that $\calO_E \to [p]_* \calO_E$ is not split.
\end{example}

\section{Some facts about derived categories}
\label{sec:poschardercat}

The purpose of this section is to record some notation and results about triangulated categories for later use.  As a general reference for triangulated categories and $t$-structures, we suggest \cite{BBD}. For the convenience of the reader, we first recall some notation regarding truncations.

\begin{notation}
Let $\calD$ be a triangulated category with a $t$-structure given by a pair $(\calD^{\geq 0}, \calD^{\leq 0})$ of full subcategories satisfying the usual axioms. For each integer $n$, we let $\calD^{\geq n} = \calD^{\geq 0}[-n]$ (respectively, $\calD^{\leq n} = \calD^{\leq 0}[-n]$); this can be thought of as the fullsubcategory spanned by objects with cohomology only in degree at least (respectively, at most) $n$. Moreover, there exist truncation functors: for each integer $n$, there exist endofunctors $\tau_{\leq n}$ and $\tau_{\geq n}$ of $\calD$ which are retractions of $\calD$ onto the fullsubcategories $\calD^{\leq n}$ and $\calD^{\geq n}$. We let $\tau_{> n} = \tau_{\geq n+1}$, and $\tau_{< n} = \tau_{\leq n-1}$. These truncation functors are not exact, but they sit in an exact triangle $\tau_{\leq n} \to \id \to \tau_{> n} \to \tau_{\leq n}[1]$. Moreover, they satisfy the adjunctions
\[ \Hom_{\calD^{\leq n}}(K,\tau_{\leq n} L) \simeq \Hom_{\calD}(K,\tau_{\leq n} L)  \simeq \Hom_{\calD}(K,L) \quad \textrm{for} \quad K \in \calD^{\leq n} \quad \textrm{and} \quad L \in \calD \]
and dually
\[ \Hom_{\calD^{\geq n}}(\tau_{\geq n} K, L)  \simeq \Hom_{\calD}(\tau_{\geq n} K,L) \simeq \Hom_{\calD}(K,L) \quad \textrm{for} \quad K \in \calD \quad \textrm{and} \quad L \in \calD^{\geq n}. \]
These adjunctions can be remembered as algebraic analogues of the fact that all maps $X \to Y$ are nullhomotopic if $X$ is an $n$-connected CW complex, and $Y$ is an $(n-1)$-truncated one.
\end{notation}

Let us fix a triangulated category $\calD$, with a $t$-structure $(\calD^{\geq 0},\calD^{\leq 0})$. The main question that arises repeatedly in the sequel is the following: given a morphism $f:K \to L$ in $\calD$ such that $H^*(f) = 0$, when can we conclude that $f = 0$?  As the non-trivial extension $\Z/2 \to \Z/2[1]$ in the derived category $\D(\Ab)$ of abelian groups shows, the short answer is ``not always''. To understand this phenomenon better, fix a test object $M \in \calD$, and consider the associated map of abelian groups 
\[\Hom(M,f):\Hom(M,K) \to \Hom(M,L)\]
The chosen $t$-structure gives rise to a functorial filtration on the morphism spaces of $\calD$ (via the filtration by cohomology groups of the target). Thus, the preceding map is a filtered map of filtered abelian groups. The assumption that $H^*(f) = 0$ implies that this filtered map induces the $0$ map on the associated graded pieces. In other words, $f$ moves the filtration one level down. This simple analysis suggests that under certain boundedness hypotheses, we may be able to salvage an implication of the form ``$H^*(f) = 0 \Rightarrow f = 0$'' at the expense of iterating a map like $f$ a few times. This idea is formalised in the following lemma:

\begin{lemma}
\label{dercatlemma}
Let $\calD$ be a triangulated category with $t$-structure $(\calD^{\geq 0},\calD^{\leq 0})$ whose heart is $\calA$. Assume that for a fixed integer $d > 0$, we are given objects $K_1,\dots,K_{d+1} \in \calD^{[1,d]}$ and maps $f_i:K_i \to K_{i+1}$ such that $H^{d+1-i}(f_i) = 0$ for all $i$. Then the composite map $f_d \circ \cdots \circ f_2 \circ f_1:K_1 \to K_d$ is the $0$ map.
\end{lemma}
\begin{proof}
Consider the exact triangle 
\[\tau_{\leq d-1} K_2 \to K_2 \to H^d(K_2)[-d] \to \tau_{\leq d-1} K_2[1].\]
Applying $\Hom_{\calD}(K_1,-)$ and using the formula (coming from adjunction)
\begin{eqnarray*}
\Hom_{\calD}(K_1,H^d(K_2)[-d]) &=& \Hom_{\calD^{\geq d}}(\tau_{\geq d}K_1,H^d(K_2)[-d]) \\
						&=& \Hom_{\calD^{\geq d}}(H^d(K_1)[-d],H^d(K_2)[-d]) \\
						&=& \Hom_{\calA}(H^d(K_1),H^d(K_2)), 
\end{eqnarray*}
we see that the map $K_1 \to H^d(K_2)[-d]$ factors through $H^d(f_1)$ and is thus $0$ by hypothesis. We may therefore choose a (non-unique) factorisation of $f_1$ of the form $K_1 \to \tau_{\leq d-1} K_2 \to K_2$. The same method shows that the morphism $\tau_{\leq d-i} K_{i+1} \to \tau_{\leq d-i} K_{i+2}$ factors through $\tau_{\leq d-(i+1)} K_{i+2}$. Thus, we obtain a diagram of morphisms:
\[ \xymatrix{
	K_1 \ar@{=}[r] \ar[d] & K_1 \ar[d]^{f_1}  \\
	\tau_{\leq d-1} K_2 \ar[r] \ar[d] & K_2 \ar[d]^{f_2} \\
	\ldots \ar[r] \ar[d] & \ldots \ar[d]^{f_d} \\
	\tau_{\leq 0} K_{d+1} \ar[r] & K_{d+1} } \]
As $K_{d+1} \in \D^{\geq 1}(\calA)$, we see that $\tau_{\leq 0} K_{d+1} = 0$. Thus, the composite vertical morphism on the left is zero, which implies that the one on the right is $0$ as well.
\end{proof}

\section{The main theorem}
 \label{sec:poscharmainthm}

This section is dedicated to the proof of Theorems \ref{killdercoh} and \ref{ddsct}. In fact, the bulk of the work involves proving Theorem \ref{killdercoh} as Theorem \ref{ddsct} then follows by a fairly formal argument. The proof we give here draws on ideas whose origin can be traced back to Hochster and Huneke's work \cite{HHBigCM} on big Cohen-Macaulay algebras in positive characteristic. We begin with a rather elementary result on extending covers of schemes. 

\begin{proposition}
\label{extendcover}
Fix a noetherian scheme $X$. Given an open dense subscheme $U \to X$ and a finite (surjective) morphism $f:V \to U$, there exists a finite (surjective) morphism $\overline{f}:\overline{V} \to X$ such that $\overline{f}_U$ is isomorphic to $f$. Given a Zariski open cover $\calU = \{j_i:U_i \to X\}$ with a finite index set, and finite (surjective) morphisms $f_i:V_i \to U_i$, there exists a finite (surjective) morphism $f:Z \to X$ such that $f_{U_i}$ factors through $f_i$. The same claims hold if ``finite (surjective)'' is replaced by ``proper (surjective)'' everywhere.
\end{proposition}
\begin{proof}
We first explain how to deal with the claims for finite morphisms. For the first part, Zariski's main theorem \cite[Th\'eor\`eme 8.12.6]{EGA4_3} applied to the morphism $V \to X$ gives a factorisation $V \hookrightarrow W \to X$ where $V \hookrightarrow W$ is an open immersion, and $W \to X$ is a finite morphism. The scheme-theoretic closure $\overline{V}$ of $V$ in $W$ provides the required compactification in view of the fact that finite morphisms are closed.

For the second part, by the first part, we may extend each $j_i \circ f_i:V_i \to X$ to a finite surjective morphism $\overline{f_i}:\overline{V_i} \to X$ such that $\overline{f_i}$ restricts to $f_i$ over $U_i \hookrightarrow X$. Setting $W$ to be the fibre product over $X$ of all the $\overline{V_i}$ is then seen to solve the problem.

To deal with the case of proper (surjective) morphisms instead of finite (surjective) , we repeat the same argument as above replacing the reference to Zariski's main theorem by one to Nagata's compactification theorem (see \cite[Theorem 4.1]{ConradNagata}).
\end{proof}

Next, we present the primary ingredient in the present proof of Theorem \ref{killdercoh}: a general technique for constructing covers to annihilate coherent cohomology of $\F_p$-schemes under suitable finiteness assumptions. The method of construction is essentially borrowed from \cite{GLBigCM} where it is used to reinterpret and simplify the ``Equational Lemma,'' one of the main ingrendients in the proof of the existence of big Cohen-Macaulay algebras in positive characteristic (see \cite{HHBigCM}).

\begin{proposition}
\label{killcohabstract}
Let $X$ be a noetherian $\F_p$-scheme with $H^0(X,\calO_X)$ finite over a ring $A$. Given an $A$-finite Frobenius-stable submodule $M \subset H^i(X,\calO_X)$ for $i > 0$, there exists a finite surjective morphism $\pi:Y \to X$ such that $\pi^*(M) = 0$
\end{proposition}
\begin{proof}
We first explain the idea informally. As $M$ is $A$-finite, it suffices to work one cohomology class at a time. If $m \in M$, then the Frobenius-stability of $M$ gives us a monic additive polynomial $g(X^p)$ such that $g(m) = 0$ where $X^p$ acts by Frobenius. After adjoining $g$-th roots of certain local functions representing a coboundary, we can promote the preceding equation in cohomology to an equation of cocycles, i.e, we find $g(\overline{m}) = 0$ where $\overline{m}$ is a cocycle of local functions that represents $m$, and the displayed equality is an equality of functions on the nose, not simply up to coboundaries. Since $g$ is monic, such functions are forced to be globally defined (after normalisation), and this gives the desired result; the details follow.

Fix a finite affine open cover $\calU = \{U_i\}$ of $X$, and consider the cosimplicial $A$-algebra  $\calC^\bullet(\calU,\calO_X)$ as a model for the $A$-algebra $\R\Gamma(X,\calO_X)$. The Frobenius action $\Frob_X^*:\R\Gamma(X,\calO_X) \to \R\Gamma(X,\calO_X)$ is modelled by the actual Frobenius map $X^p:x \mapsto x^p$ on each term of this $A$-algebra. This gives $\calC^\bullet(\calU,\calO_X)$ the structure of an $A\{X^p\}$-module, where $A\{X^p\}$ is the non-commutative polynomial ring on one generator $X^p$ over $A$ satisfying the commutation relation $r^pX^p = X^pr$ (see \cite[\S 1.1]{LaumonCDMV1} for more details on this ring). In more concrete terms, at the level of cohomology, we see the following: for each polynomial $g \in A\{X^p\}$, classes $\alpha,\beta \in H^i(X,\calO_X)$, and a scalar $r \in A$, we have $g(\alpha + \beta) = g(\alpha) + g(\beta)$, and $g(r \alpha) = r^p g(\alpha)$.

The $A$-finiteness of the Frobenius stable module $M$ ensures that for any class $m \in M$, there exists a {\em monic} polynomial $g \in A\{X^p\}$ such that $g(m) = 0$. If we pick representatives in $\calC^\bullet(\calU,\calO_X)$ for this equation, we obtain an equation in $\calC^i(\calU,\calO_X)$ of the form 
\[g(\tilde{m}) = d(n) \]
where $\tilde{m} \in \calC^i(\calU,\calO_X)$ is a cocycle lifting $m$ and $n \in \calC^{i-1}(\calU,\calO_X)$. As $g$ is a monic equation, we can find a finite surjective morphism $\pi':Y' \to X$ such that $n = g(n')$ for some $n' \in \calC^i(\calU \times_X Y',\calO_{Y'})$. For example, we could do the following: for each component $n_j$ of $n$ (where $j$ is a multi-index), the scheme $V_j = \Spec(\calO(U_j)[T]/(g(T) - n_j))$ is a quasi-finite $X$-scheme such that the equation $g(n') = n$ admits a solution in $H^0(V_j,\calO_{V_j})$. Using Proposition \ref{extendcover}, we find a $Y'$ and $n'$ with the desired properties. The additivity of Frobenius now tells us that we obtain an equation in $\calC^i(\calU \times_X Y',\calO_{Y'})$ of the form 
\[g(\tilde{m} - d(n')) = 0.\]
The monicity of $g$ implies that the components of $\tilde{m} - d(n')$ are integral over $A$. Setting $Y$ to be an irreducible component of $Y' \times_{\Spec(A)} \Spec(A[T]/(g(T)))$ that dominates $Y'$ under the natural map, we find a finite surjective morphism $Y \to Y'$. The pullback of $\tilde{m} - d(n')$ in $\calC^i(\calU \times_X Y,\calO_Y)$ is a vector of local functions whose components satisfy the monic polynomial $g$ over $A$.  As $Y$ is integral and $H^0(Y,\calO_Y)$ already contains roots of $g$, it follows that these functions are globally defined. Thus, they lie in the image of the natural map $H^0(Y,\calO_Y) \to \calC^\bullet(\calU \times_X Y,\calO_Y)$ where $H^0(Y,\calO_Y)$ is viewed as a constant cosimplicial algebra. As the complex underlying the former cosimplicial algebra has cohomology only in degree $0$, it follows that $\tilde{m} - d(n')$ is a coboundary, which implies that $\tilde{m}$ is a coboundary on $Y$, which shows that $Y$ satisfies the required conditions.
\end{proof}

\begin{remark}
One may wonder whether Proposition \ref{killcohabstract} can be refined to show the existence of generically separable finite surjective maps that kill the relevant cohomology groups. In the local algebra setting, one can indeed do so by \cite[Theorem 1.3]{SinghSannai}.  Globally, however, requiring separability is too strong. For example, if $X$ is a smooth projective variety over a perfect field $k$ with $\alpha \in H^1(X,\calO_X)$ a non-zero class killed by $\Frob_X$, then for any finite surjective generically separable map $\pi:Y \to X$, one has $\pi^* \alpha \neq 0$; see \cite[Lemma 5]{MumfordPathologies3} for a proof.
\end{remark}

\begin{remark}
Proposition \ref{killcohabstract} is proven by above by mimicking the cocycle-theoretic methods of \cite{GLBigCM}. It is also possible to give more conceptual proofs of this result. We refer the reader to \cite{Bhattanngrpsch} for a proof based on general results on finite flat group schemes, and \cite{Bhattmixedcharpdiv} for a geometric proof based on curve fibrations which has the advantage of generalising to mixed characteristic.
\end{remark}

As a corollary of Proposition \ref{killcohabstract} and the finiteness properties enjoyed by proper morphisms, we arrive at the following result:

\begin{corollary}
\label{killcoh}
Let $f:X \to S$ be a proper morphism of $\F_p$-schemes, with $S$ noetherian and affine. Then there exists a finite surjective morphism $\pi:Y \to X$ such that $\pi^*:H^i(X,\calO_X) \to H^i(Y,\calO_Y)$ is $0$ for $i > 0$.
\end{corollary}
\begin{proof}
The properness of $X$ over an affine implies that $H^i(X,\calO_X)$ is a finite $H^0(X,\calO_X)$-module and that $H^i(X,\calO_X) = 0$ for $i$ sufficiently large (see \cite[Corollaire 3.2.3]{EGA3_1}). Proposition \ref{killcohabstract} then finishes the proof.
\end{proof}
We will now finish the proof of Theorem \ref{killdercoh}. To pass from the conclusion of Corollary \ref{killcoh} to the general statement of Theorem \ref{killdercoh}, the obvious strategy is to cover $S$ with affines, construct covers that work over the affines, and take the normalisation of $X$ in the fibre product of all of these. When carried out, this process produces a finite cover $\pi:Y \to X$ such that, with $g = f \circ \pi$, the maps $R^i f_*\calO_X \to R^i g_*\calO_Y$ are $0$ for $i > 0$. This is not quite enough to prove the theorem: a map in $\D(\Coh(S))$ that induces the $0$ map on cohomology sheaves is not necessarily zero. However, with the boundedness conditions enforced by properness, a sufficiently high iteration of this process turns out to be enough.

\begin{proof}[Proof of Theorem \ref{killdercoh}]
Fix a finite affine covering $\calU = \{U_i\}$ of $S$, and denote $X \times_S U_i$ by $X_i$. Using Corollary \ref{killcoh}, we can find finite surjective maps $\phi_i:Z_i \to X_i$ such that the induced map $H^j(X_i,\calO_{X_i}) \to H^j(Z_i,\calO_{Z_i})$ is $0$ for each $j > 0$. Using Proposition \ref{extendcover}, we may find a finite surjective morphism $\phi:Z \to X$ such that $\phi_{U_i}$ factors through $\phi_i$. This implies that $\R^j f_*\calO_X \to \R^j (f \circ \phi)_* \calO_Z$ is $0$ for each $j$ (as vanishing is a local statement on $S$). Iterating this construction $\dim(X)$ times and using Lemma \ref{dercatlemma}, we obtain a proper $S$-scheme $g:Y \to S$ and a finite surjective $S$-morphism $\pi:Y \to X$ such the natural pullback  map $\pi^*:\tau_{\geq 1} \R f_*\calO_X \to \tau_{\geq 1} \R g_*\calO_Y$ is $0$, thereby proving the theorem.
\end{proof}

Finally, having proven Theorem \ref{killdercoh}, we point out how Theorem \ref{ddsct} follows. 

\begin{proof}[Proof of Theorem \ref{ddsct}]
It is clear that all D-splinters are also splinters. Conversely, let $S$ be splinter over $\F_p$, and let $f:X \to S$ be a proper surjective morphism. By Theorem \ref{killdercoh}, there exists a finite surjective morphism $\pi:Y \to X$ such that, with $g = f \circ \pi$, the pullback map $\tau_{\geq 1} \R f_*\calO_X \to \tau_{\geq 1} \R g_*\calO_Y$ is $0$. By applying $\Hom(\R f_*\calO_X,-)$ to the exact triangle
\[ g_* \calO_Y \to \R g_*\calO_Y \to \tau_{\geq 1} \R g_*\calO_Y \to g_*\calO_Y[1] \]
we see that the natural pullback map $\R f_*\calO_X \to \R g_*\calO_Y$ factors through $g_*\calO_Y \to \R g_*\calO_Y$; choose some factorisation $s:\R f_* \calO_X \to g_* \calO_Y$. As $g:Y \to S$ is a proper surjective morphism, the algebra $g_* \calO_Y$ is a coherent sheaf of algebras corresponding to the structure sheaf of a finite surjective morphism. By assumption, the natural map $\calO_S \to g_*\calO_Y$ has a splitting $s$, and thus the map $s \circ t$ splits $\calO_S \to \R f_*\calO_X$.
\end{proof}

\section{Some refinements}
\label{sec:poscharcommentary}

Roughly speaking, Theorem \ref{killdercoh} says that proper morphisms behave like finite morphisms after passage to finite covers, at least as far as theorems concerning the annihilation of coherent sheaf cohomology. In the following proposition, we formalise this intuition, extract a kind of ``converse'' to this statement, and work with non-trivial coefficients. These results will be useful in the sequel when we prove vanishing results. 

\begin{proposition}
\label{killwithh}
Let $S$ be a noetherian $\F_p$-scheme, and $f:X \to S$ be a proper surjective morphism. Then we can find a diagram
\[ \xymatrix{ Y \ar[r]^\pi \ar[d]^a \ar[rd]^g & X \ar[d]^f \\ S' \ar[r]^h & S } \]
with $\pi$ and $h$ finite surjective morphisms such that for every locally free sheaf $\calM$ on $S$ and every $i \geq 0$, we have:
\begin{enumerate}
\item The morphism $h^*:H^i(S,\calM) \to H^i(S',h^*\calM)$ factors through $f^*:H^i(S,\calM) \to H^i(X,f^*\calM)$.
\item The morphism $\pi^*:H^i(X,f^*\calM) \to H^i(Y,g^*\calM)$ factors through $a^*:H^i(S',h^*\calM) \to H^i(Y,g^*\calM)$.
\end{enumerate}
\end{proposition}
\begin{proof}
Theorem \ref{killdercoh} gives a finite surjective morphism $\pi:Y \to X$ such that, with $g = f \circ \pi$, we have a map $s$ and the following diagram:
\[ \xymatrix{
	\calO_S \ar[r] \ar[d] & g_*\calO_Y \ar[d] \\
	\R f_*\calO_X \ar[r] \ar[ur]^s & \R g_*\calO_Y } \]
We claim that this is a commutative diagram. The triangle based at $\R g_*\calO_Y$ commutes by construction. Given this commutativity, to see that the triangle based at $\calO_S$ commutes, it suffices to show that $\Hom(\calO_S,g_*\calO_Y) \to \Hom(\calO_S,\R g_*\calO_Y)$ is injective. This injectivity (and, in fact, bijectivity) follows from adjunction for $\tau_{\leq 0}$. Thus, the preceding diagram is a commutative diagram in $\D(\Coh(S))$. Applying $- \otimes \calM$, setting $S'$ to be the Stein factorisation of $Y \to S$, and using the projection formula now gives the desired result.
\end{proof}

We have not strived to find the most general setting for Theorem \ref{killdercoh}. For example, one can easily extend the theorem to algebraic spaces or even Deligne-Mumford stacks. On the other hand, the properness hypothesis seems essential as the example below shows. In fact, the method of the proof shows that the essential property we use is that the relative cohomology classes of the structure sheaf for $f:X \to S$ are annihilated by a monic polynomial in Frobenius. We do not know if there is a better characterisation of this class of maps.

\begin{example}
\label{needproper}
Fix a base field $k$. Let $X = \A^2$, and $U = \A^2 - \{0\}$. The quotient map $U \to U/\G_m = \P^1$ gives a natural identification $H^1(U,\calO_U) = \oplus_{i \in \Z} H^1(\P^1,\calO(i))$. We claim that the non-zero classes in this group cannot be killed by a finite cover of $U$. To see this, note that one may view $H^1(U,\calO_U)$ as the local cohomology group $H^2_{\{0\}}(X,\calO_U) = H^2_\fm(R)$, where $R = k[x,y]$ is the coordinate ring of $X$ and $\fm = (x,y)$ is the maximal ideal corresponding to the origin. Given a finite surjective morphism $\pi:Y \to U$, we may normalise $X$ in $\pi$ to obtain a finite surjective morphism $\overline{\pi}:\overline{Y} \to X$ which realises $\pi$ as the fibre over $U$. As before, the cohomology group $H^1(Y,\calO_Y)$ can be viewed as $H^2_{\overline{Y} \setminus Y}(\overline{Y},\calO_{\overline{Y}})$ which, in turn, may be viewed as $H^2_\fm(S)$, where $S$ is the coordinate ring of $\overline{Y}$ considered as an $R$-module in the natural way. Under these identifications, the pullback map $H^1(U,\calO_U) \to H^1(Y,\calO_Y)$ corresponds to the morphism $H^2_\fm(R) \to H^2_\fm(S)$ induced by the inclusion $R \to S$ coming from $\overline{\pi}$. By Example \ref{ex:smdsc}, the inclusion $R \to S$ is a direct summand as an $R$-module map. In particular, the map $H^2_\fm(R) \to H^2_\fm(S)$ is injective, which shows that the non-zero classes in $H^1(U,\calO_U)$ persist after passage to finite covers.
\end{example}

\section{Application: A result in commutative algebra}
\label{sec:poscharappalg}

We discuss some applications of Proposition \ref{killcohabstract} to commutative algebra. Most of these applications are implicit in \cite{GLBigCM}. The first result we want to dicuss is an analogue of Proposition \ref{killcohabstract} for local cohomology.

\begin{proposition}
\label{killcohlocal}
Let $(R,\fm)$ be an excellent local noetherian $\F_p$-algebra such that $R$ is finite over some ring $A$. For any $A$-finite Frobenius-stable submodule $M \in H^i_\fm(R)$ with $i \geq 1$, there exists a finite surjective morphism $f:\Spec(S) \to \Spec(R)$ such that $f^*(M) = 0$.
\end{proposition}
\begin{proof}
Since $R$ is excellent, we may pass to the normalisation and assume that $R$ is normal. In particular, $H^i_\fm(R) = 0$ for $i = 0,1$. For $i > 1$, we have an Frobenius equivariant identification $\delta:H^{i-1}(U,\calO_U) \simeq H^i_\fm(R)$, where $U = \Spec(R) - \{\fm\}$ is the punctured spectrum of $R$. Since $i > 1$, Proposition \ref{killcohabstract} gives us a finite surjective morphism $f:V \to U$ such that $f^*(\delta^{-1}(M)) = 0$. Setting $S$ to be the normalisation of $R$ in $V$ is then easily seen to do job.
\end{proof}

Next, we dualise the Proposition \ref{killcohlocal} to obtain a global result in terms of dualising sheaves.

\begin{proposition}
\label{killcohdual}
Let $X$ be an excellent noetherian $\F_p$-scheme of equidimension $d$ that admits a dualising complex $\omega_X^\bullet$. Then there exists a finite surjective morphism $\pi:Y \to X$ such that $\tau_{> -d}(\Tr_\pi) = 0$, where $\Tr_\pi$ is the trace map $\Tr_\pi:\pi_*\omega^\bullet_Y \to \omega^\bullet_X$.
\end{proposition}
\begin{proof}
Fix an integer $i > 0$. We will prove by induction on the dimension $d = \dim(X)$ that there exists a finite surjective morphism $\pi':Y' \to X$ such that $\calH^{-d+i}(\Tr_{\pi'}) = 0$; this is enough by virture of Lemma \ref{dercatlemma} and the fact that the dualising complexes appearing have bounded amplitude. We may assume that $d > 0$ as the there is nothing to prove when the dimension is $0$. By passing to irreducible components, we may even assume that $X$ is integral. As vanishing of a map of sheaves is a local statement, we reduce to the case that $X$ is an excellent noetherian local $\F_p$-domain $(R,\fm)$ admitting a dualising complex. For each non-maximal $\fp \in \Spec(R)$, we can inductively find a finite morphism $\pi_\fp:Y_\fp \to \Spec(R_\fp)$ such that $\calH^{-d_{R_\fp} + i}(\Tr_{\pi_\fp})$ is the $0$ map. By duality formalism, the $R$-module $\calH^{-d + i}(\omega_R^\bullet)$ localises to $\calH^{-d_{\R_\fp} + i}(\omega_{R_\fp}^\bullet)$ at $\fp$. Hence, the normalisation $\overline{\pi_\fp}:\overline{Y_\fp} \to X$ induces the $0$ map on $\calH^{-d + i}(\Tr_\pi)$ when localised at $\fp$. Finding such a cover for each non-maximal prime $\fp$ in the finite set of associated primes of $\calH^{-d+i}(\omega_R^\bullet)$ and normalising $X$ in the fibre product of the resulting collection, we find a cover $\pi:Y \to X$ such that $\calH^{-d+i}(\Tr_\pi)$ has an image supported only at the closed point. Setting $Y = \Spec(S)$, duality tells us that the image $M$ of $H^{d-i}_\fm(R) \to H^{d-i}_\fm(S)$ is a finite length Frobenius-stable $R$-submodule. Proposition \ref{killcohlocal} then allows us to find a finite surjective morphism $g:\Spec(T) \to \Spec(S)$ such that $g^*(M) = 0$. It follows that the composite map $\pi':\Spec(T) \to \Spec(R)$ induces the $0$ map on $H^{d-i}_\fm(R)$. By duality, we see that $\calH^{-d+i}(\Tr_{\pi'}) = 0$ as desired.
\end{proof}

Using Proposition \ref{killcohlocal}, we discover that splinters are automatically Cohen-Macaulay. 

\begin{corollary}
\label{dsciscm}
Let $(R,\fm)$ is an excellent noetherian local $\F_p$-algebra that is a splinter. Assume that $R$ admits a dualising complex. Then $R$ is a normal Cohen-Macaulay domain.
\end{corollary}
\begin{proof}
The normality of $R$ follows from the argument in Example \ref{ex:char0}. To verify that $R$ is Cohen-Macaulay, it suffices to show that $\omega^\bullet_R$ is concenctrated in degree $d$ where $d = \dim(R)$, i.e., that $\calH^{-d+k}(\omega_R^\bullet) = 0$ for $k >  0$. By Proposition \ref{killcohdual}, we can find a finite surjective morphism $\pi:\Spec(S) \to \Spec(R)$ such that $\calH^{-d+k}(\Tr_\pi) = 0$, where $\Tr_\pi:\pi_*\omega_S^\bullet \to \omega_R^\bullet$. Since $R$ is a splinter, the inclusion $R \to S$ is a direct summand. Applying $\R\calHom(-,\omega_R^\bullet)$, we see that the trace map $\Tr_\pi$ is the projection onto a summand. Hence, the assumption that $\calH^{-d+k}(\Tr_\pi) = 0$ implies that $\calH^{-d+k}(\omega_R^\bullet) = 0$, as desired.
\end{proof}

\begin{remark}
One key ingredient in the proof of Proposition \ref{killcohdual} is the good behaviour of local cohomology and dualising sheaves with respect to localisation. This behaviour seems to first have been observed by Grothendieck in \cite[Expos\'e VIII, Th\'eor\`eme 2.1]{SGA2} where it is used to show the following: a noetherian local ring $(R,\fm)$ of dimension $d$ that is Cohen-Macaulay outside the closed point and admits a dualising complex has the property that $H^i_\fm(R)$ has finite length for $i < d$. This argument can also be found in the main theorem \cite{GLBigCM}.
\end{remark}

\section{Application: A question of Karen Smith}
\label{sec:poscharappksmith}

The main result of Hochster-Huneke \cite{HHBigCM} is a result in commutative algebra. While geometrising it in \cite{KSmithVanishing}, K. Smith arrived at the following question (see \cite{KSmithVanishingErr}):

\begin{question}
\label{smithconj}
Let $X$ be a projective variety over a field $k$ of characteristic $p$, and let $\calL$ be a ``weakly positive'' line bundle on $X$. For any $n \in \Z$ and any $0 < i < \dim(X)$, does there exist a finite surjective morphism $\pi:Y \to X$ such that $H^i(X,\calL^{\otimes n}) \to H^i(Y,\pi^*\calL^{\otimes n})$ is $0$?
\end{question}

Using the algebraic result of Hochster-Huneke \cite{HHBigCM}, one can show that if we take ``weakly positive'' to mean ample, then Question \ref{smithconj} has an affirmative answer (see Remark \ref{hhforample}). Smith had originally hoped that ``weakly positive'' could be taken to mean nef. We give some examples in the sequel to show that this cannot be the case. However, first, we prove some positive results using the theorems above.

\subsection{Positive results}

We first examine Question \ref{smithconj} in the case of positive twists. It is clear that being ample is a sufficiently positive condition for the required vanishing statement to be true: Frobenius twisting can be realised by pulling back along a finite morphism and has the effect of changing $\calL$ by $\calL^{\otimes p}$, whence Serre vanishing shows the desired result. It is natural to wonder if the result passes to the closure of the ample cone, i.e., the nef cone. We show in Example \ref{nefnotenough} that this is not the case: there exist non-torsion degree $0$ line bundles on surfaces whose middle cohomology cannot be killed by finite covers. On the other hand, Corollary \ref{killcoh} coupled with the fact that torsion line bundles can be replaced with $\calO$ on passage to a finite cover ensures that Question \ref{smithconj} has a positve answer for torsion line bundles. The necessity of the non-torsion requirement and the observation that torsion line bundles are semiample suggests the following proposition.

\begin{proposition}
\label{smithconjpos}
Let $X$ be a proper variety over a field of characteristic $p$, and let $\calL$ be a semiample line bundle on $X$. For any $i > 0$, there exists a finite surjective morphism $\pi:Y \to X$ such that the induced map $H^i(X,\calL) \to H^i(Y,\pi^*\calL)$ is $0$.
\end{proposition}
\begin{proof}
As $\calL$ is a semiample bundle, there exists some positive integer $m$ such that $\calL^{\otimes m}$ is globally generated. If we fix a basis $s_1,\dots,s_k$ for $H^0(X,\calL^{\otimes m})$, then the cyclic covering trick (see \cite[Proposition 4.1.3]{LazPos1}) ensures that there's a finite flat cover $\pi:\tilde{X} \to X$ such that $\pi^*(s_i)$ admits an $m$-th root in $H^0(\tilde{X},\pi^*\calL)$ and, consequently, $\pi^*\calL$ is globally generated. In particular, as semiamplitude is preserved under pullbacks, we may replace $X$ with $\tilde{X}$ and assume that $\calL$ arises as the pullback of an ample bundle $\calM$ under a proper surjective morphism $f:X \to S$. Furthermore, once $f:X \to S$ is fixed, to show the required vanishing statement, we may always replace $\calL$ by $\calL^{\otimes p^j}$ for $j \gg 0$  because the Frobenius morphism $\Frob_X:X \to X$ is finite surjective with $\Frob_X^*\calL = \calL^{\otimes p}$. Now the projection formula for $f$ implies that $\R f_*(\calL^{\otimes p^j}) = \R f_*\calO_X \otimes^\L_S \calM^{\otimes p^j}$. Using Theorem \ref{killdercoh}, we may find a finite surjective morphism $\pi:Y \to X$ such that, with $g = f \circ \pi$, we have a factorisation $\R f_*(\calL^{\otimes p^j}) \to g_*f^*(\calL^{\otimes p^j}) \to \R g_*\pi^*(\calL^{\otimes p^j})$ of the natural map $\pi^*:\R f_*(\calL^{\otimes p^j}) \to \R g_*\pi^*(\calL^{\otimes p^j})$. Applying  $H^i(S,-)$ to the composite morphism gives us the desired morphism. Thus, to show the required statement, it suffices to show that $H^i(S,g_*\pi^*(\calL^{\otimes p^j})) = 0$ for $j \gg 0$.  By the projection formula, we have 
\[H^i(S,g_*\pi^*(\calL^{\otimes p^j})) = H^i(S,g_*g^*(\calM^{\otimes p^j})) = H^i(S,g_*\calO_Y \otimes \calM^{\otimes p^j})\]
As $\calM$ is ample, this group vanishes by Serre vanishing for $j \gg 0$, as required.
\end{proof}

Based on Proposition \ref{smithconjpos}, one might expect that semiamplitude is a positive enough property for Question \ref{smithconj} to have an affirmative answer for the case of negative twists as well. We show in Example \ref{sanotenough} that this is not true; the key feature of that example is that the semiample line bundle defines a map that is not generically finite. In fact, this feature is essentially the only obstruction: if $\calL$ is both semiample and big, then Question \ref{smithconj} has an affirmative answer even for negative twists of $\calL$.

\begin{proposition}
\label{smithconjneg}
Let $X$ be a proper variety over a field of characteristic $p$, and let $\calL$ be a semiample and big line bundle on $X$. For any $i < \dim(X)$, we can find a finite surjective morphism $\pi:Y \to X$ such that the induced map $H^i(X,\calL^{-1}) \to H^i(Y,\pi^*\calL^{-1}) $ is $0$
\end{proposition}
\begin{proof}
We first describe the idea informally. Using Proposition \ref{killwithh} and arguments similar to those in the proof of Proposition \ref{smithconjpos}, we will reduce to the case that $\calL$ is actually {\em ample} on $X$. In this case, we give a direct proof using Proposition \ref{killcohdual}; the details follow.

Fix an integer $i < \dim(X)$. As $\calL$ is big, there is nothing to show for $i = 0$ and, thus, we may assume $i > 0$. As in the proof of Proposition \ref{smithconjpos}, at the expense of replacing $X$ by a finite flat cover, we may assume that $\calL$ arises as the pullback of an ample line bundle $\calM$ under a proper surjective morphism $f:X \to S$. As bigness is preserved under passage to finite flat covers, we may continue to assume that $\calL$ is big. In particular, the map $f$ is forced to be an alteration.  By Proposition \ref{killwithh}, we can find a diagram
\[ \xymatrix{ Y \ar[r]^\pi \ar[d]^a \ar[rd]^g & X \ar[d]^f \\
			S' \ar[r]^h & S } \]
with $\pi$ and $h$ finite surjective, such that we have a factorisation $H^i(X,\calL^{-1}) \stackrel{s}{\to} H^i(S',h^*\calM^{-1}) \stackrel{a^*}{\to} H^i(Y,\pi^*\calL^{-1})$ of $\pi^*$ for some map $s$. Moreover, given a finite cover $b:S'' \to S$, we can form the diagram 
\[ \xymatrix{ Y \times_{S'} S'' \ar[r]^-{\pr_1} \ar[d]^{\pr_2} & Y \ar[rd]^g \ar[d]^a \ar[r]^\pi & X \ar[d]^f \\
			  S'' \ar[r]^b & S' \ar[r]^h & S } \]
This means that at the level of cohomology, we have a commutative diagram
\[ \xymatrix{ H^i(X,\calL^{-1}) \ar[r]^-{ (\pi \circ \pr_1)^* } \ar[d]^s & H^i(Y \times_{S'} S'', (\pi \circ \pr_1)^*\calL^{-1}) \\
				H^i(S',h^*\calM^{-1}) \ar[r]^{b^*} & H^i(S'',b^*h^*\calM^{-1}) \ar[u]^{\pr_2^*} } \]
Thus, it suffices to show that $H^i(S',h^*\calM^{-1})$ can be killed by finite covers of $S'$. As $h$ is a finite morphism, the bundle $h^*\calM$ is ample. That $f$ was an alteration forces $\dim(S') = \dim(X)$ and, therefore, $0 < i < \dim(S')$. In other words, we are reduced to verifying the claim in the theorem under the additional assumption that $\calL$ is ample.

As we are free to replace $X$ by a Frobenius twist (which increases the positivity of $\calL$), we may assume that $\calL$ has the property that $H^j(X,\calL \otimes \omega_X) = 0$ for all $j > 0$, where $\omega_X$ is the dualising {\em sheaf} on $X$. Now choose a finite surjective morphism $\pi:Y \to X$ satisfying the conclusion of Proposition \ref{killcohdual}. With $d = \dim(X)$, the trace map induces the following morphism of triangles in $\D^b(\Coh(X))$:
\[ \xymatrix{ \pi_* \omega_Y[d] \ar[r] \ar[d]^-a & \pi_* \omega_Y^\bullet \ar[r] \ar[d]^-b \ar@{.>}[ld]^-s & \tau_{> -d} \pi_*\omega_Y^\bullet \ar[r] \ar[d]^-{c = 0} & \pi_* \omega_Y[d+1] \ar[d] \\
			  \omega_X[d] \ar[r] & \omega_X^\bullet \ar[r] & \tau_{> -d} \omega_X^\bullet \ar[r] & \omega_X[d+1]. } \]
Here $s$ is a map that whose existence is ensured by the equation $c = 0$ (but $s$ is not necessarily unique). Tensoring this diagram with $\calL$, using the flatness of $\calL$, and using the projection formula gives us the following morphism of triangles:
\[ \xymatrix{ \pi_* (\omega_Y \otimes \pi^*\calL) [d] \ar[r] \ar[d]^-{a_\calL} & \pi_* (\omega_Y^\bullet \otimes \pi^*\calL) \ar[r] \ar[d]^-{b_\calL} \ar@{.>}[ld]^-{s_\calL}& \tau_{> -d} \pi_*(\omega_Y^\bullet \otimes \pi^*\calL)\ar[r] \ar[d]^-{c_\calL = 0} & \pi_* (\omega_Y \otimes \pi^*\calL)[d+1] \ar[d] \\
			  \omega_X \otimes \calL [d] \ar[r] & \omega_X^\bullet \otimes \calL \ar[r] & \tau_{> -d} \omega_X^\bullet \otimes \calL \ar[r] & \omega_X \otimes \calL [d+1]. } \]
The commutativity of the above diagram and existence of $s_{\calL}$ shows that for any integer $i$, the image of the natural trace map
\[ H^{-i}(b_\calL):H^{-i}(Y,\omega_Y^\bullet \otimes \pi^*\calL) \to H^{-i}(X,\omega_X^\bullet \otimes \calL)\]
lies in the image of the natural map
\[ H^{d-i}(X,\omega_X \otimes \calL) = H^{-i}(X,\omega_X \otimes \calL[d]) \to H^{-i}(X,\omega_X^\bullet \otimes \calL). \]
Now choose $i$ such that $0 < i < d$. By assumption, the source of the preceding map is then trivial. Hence, we find that the map $H^{-i}(b_\calL)$ is also $0$. Dualising, it follows that
\[\pi^*:H^i(X,\calL^{-1}) \to H^i(Y,\pi^*\calL^{-1}) \]
is trivial, as desired.
\end{proof}

\begin{remark}
\label{hhforample}
Consider the special case of Proposition \ref{smithconjneg} when $\calL$ is ample. We treated this case directly in the second half of the proof above using Proposition \ref{killcohdual}. It is possible to replace this part of the proof by a reference to \cite[Theorem 1.2]{HHBigCM}, the main geometric theorem of that paper. We have not adopted this approach as we feel that the proof given above using Proposition \ref{killcohdual} is cleaner than the algebraic approach of \cite{HHBigCM} which involves developing a theory of graded integral closures (see \cite[\S 4]{HHBigCM}) and reducing to a local algebra theorem.
\end{remark}

\begin{remark}
\label{delillrmk}
Proposition \ref{smithconjneg} can be viewed as a weaker version of Kawamata-Viehweg vanishing in characteristic $p$ up to finite covers. The natural way to approach this question is to ask for liftability to the $2$-truncated Witt vector ring $\W_2(k)$ by smooth varieties up to finite covers, and then quote Deligne-Illusie \cite{DeligneIllusie}. In fact, by Proposition \ref{killwithh}, it would suffice to show that any variety can be dominated by a smooth one that lifts to $\W_2(k)$. Unfortunately, we do not know the answer to this.
\end{remark}

\begin{remark}
The results proven in this section were used in \cite{BSTFSingAlt} to obtain a unified description of test and multiplier ideals via alterations.
\end{remark}

\subsection{Counterexamples}
\label{subsec:countex}

This section is dedicated to providing the examples promised earlier. The most important examples here are Examples \ref{sanotenough} and \ref{nefbignotenough}; the former shows that the middle cohomology of the inverse of a semiample line bundles on a smooth projective variety cannot always be killed by finite covers, while the latter shows that the middle cohomology of a nef and big line bundle on a smooth projective variety cannot be killed by finite covers. 

We start off by constructing a degree $0$ line bundle on a {\em singular} stable curve whose cohomology cannot be killed by Frobenius; we will later use this to construct the smooth examples promised above.

\begin{example}
\label{ex:stablecurve}
Let $k$ be a perfect field of characteristic $p$ over $\F_p$. Let $E$ be an elliptic curve over $k$ with identity $z \in E(k)$, and let $\calL$ be a degree $0$ line bundle on $E$ that is not $p$-power torsion (assumed to exist). Let $C = E \sqcup_z E$ be the stable genus $2$ curve obtained by glueing $E$ to itself at $z$, and let $\calM \in \Pic(C)$ be the line bundle obtained by glueing $\calL$ over each copy of $E$ to itself using the identity isomorphism $\calL|_z \simeq \calL|_z$. Then we claim that $H^1(C,\calM)$ is a one-dimensional $k$-vector space, and $\Frob_C^e:C \to C$ induces an isomorphism $\Frob_C^{e,*}:H^1(C,\calM) \to H^1(C,\calM^{p^e})$ for each $e > 0$.  To see these claims, note that we have an exact triangle
\[ \R\Gamma(C,\calM) \to \R\Gamma(E,\calL) \oplus \R\Gamma(E,\calL) \to \R\Gamma(z,\calL|_z) \]
where $z$ abusively denotes the reduced subscheme structure on the point $z$. Since $\calL$ is chosen to be not $p$-power torsion, we have $\R\Gamma(E,\calL) \simeq \R\Gamma(E,\calL^{p^e}) = 0$ for each $e > 0$. Hence, the triangle above (and a similar one for the Frobenius pullback) degenerates to give a diagram
\[ \xymatrix{ H^0(z,\calL|_z) \ar[r]^{\simeq} \ar[d]^{\Frob_z^{e,*}} & H^1(C,\calM) \ar[d]^{\Frob_C^{e,*}} \\
			  H^0(z,\calL^{p^e}|_z) \ar[r]^{\simeq} & H^1(C,\calM^{p^e}). } \]
One can check that the left vertical map is an isomorphism (it is given on sections by $s \mapsto s^{\otimes p^e}$), and thus so is the right vertical one. We remark that this construction can be adapted to work for arbitrary genera by glueing in more copies of $E$.
\end{example}

Next, we use Example \ref{ex:stablecurve} to construct a degree zero line bundle $\calM$ on a smooth curve $C$ whose top cohomology cannot be killed by finite covers of $C$. This example shows that the semiamplitude hypothesis in Proposition \ref{smithconjpos} cannot be weakened to a nefness hypothesis. 

\begin{example}
\label{nefnotenoughfortop}
Let $C$ be a curve of genus $g(C) \geq 2$ over an uncountable algebraically closed field $k$ of characteristic $p$, and let $\calM$ be a degree $0$ line bundle on $C$ with the following properties:
\begin{enumerate}
\item For each integer $e > 0$, the Frobenius map $\Frob_C^e:C \to C$ induces an injective map $\Frob_C^{e,*}:H^1(C,\calM) \to H^1(C,\calM^{p^e})$.
\item For each integer $e > 0$, the line bundle $\calM^{-p^e}$ does not occur as a subquotient of $h_* \calO_{C'}$ for any finite \'etale cover $h:C' \to C$.
\end{enumerate}
We will show that such pairs $(C,\calM)$ exist over a sufficiently large field, and that for any finite surjective map $\pi:C' \to C$, the pullback $H^1(C,\calM) \to H^1(C',\pi^* \calM)$ is injective. Our proof of existence is based on the deformation theory of stable curves (see \cite{DeligneMumford}) coupled with Example \ref{ex:stablecurve}, while the proof of injectivity uses the Harder-Narasimhan filtration for vector bundles on a curve (see \cite[\S 6.4.A]{LazPos2}).

Existence: Consider the moduli stack\footnote{One can avoid all mention of stacks in this proof by simply taking $\calX$ to be a versal family of stable curves equipped with line bundles of degree $0$.} $\calX := \overline{\calM_g}(B\G_m)$ parametrising pairs $(C,\calM)$ where  $C$ is a  {\em stable} genus $g$ curve, and $\calM$ is a line bundle on $C$. This stack is smooth by deformation theory: stable curves have unobstructed deformations by \cite[Corollary 5.32]{IllusieFormalGeom}, and line bundles on any proper curve are unobstructed as $H^2(C,\calO_C) = 0$. Let $\pi:\calC \to \calX$ denote the universal curve, and consider the usual diagram:
\[ \xymatrix{ \calC \ar@/_/[dddr]_{\pi} \ar@/^/[ddrrr]^{\Frob_\calC} \ar[ddr]^-{\Frob_\pi} & & \\ 
			 & & \\
			  & \calC^{(1)} \ar[d]^-{\pi^{(1)}} \ar[rr]^-{\Frob_\calX} & & \calC \ar[d]^-{\pi} \\
			  & \calX \ar[rr]^-{\Frob_\calX} & & \calX.} \]
If $\calM^\univ \in \Pic(\calC)$ denotes the universal line bundle, then the relative Frobenius map  $\Frob_\pi$ induces a morphism $a:\R^1 \pi^{(1)}_* \Frob_\calX^* \calM^\univ \to \R^1 \pi_* (\calM^\univ)^{\otimes p}$. Now consider the open substack $U \subset \calX$ where the formation of $\R^i \pi_* \calM^\univ$ commutes with base change. The fibre of $a$ over a point $(C,\calM) \in U$ is given by the usual Frobenius map $H^i(C,\calM) \to H^i(C,\calM^p)$. Since the curve $(C,\calM)$ constructed in Example \ref{ex:stablecurve} lies in this open $U$, we see that the map $a$ has a trivial kernel in a Zariski open neighbourhood of this particular point. Applying the same analysis to higher Frobenius twists and then intersecting, we find that for a very general $(C,\calM)$ in $\calX$ with $C$ smooth, the first property above is satisfied.  For the second property above, note that for each finite \'etale cover $h:C' \to C$, the bundle $h_* \calO_{C'}$ is a semistable degree $0$ vector bundle: this can be checked after finite \'etale base change on $C$, but then it is clear as $C'$ splits completely over some finite \'etale cover. As the category of semistable vector bundles of degree $0$ is an artinian and noetherian $k$-linear abelian category, it follows that only finitely many degree $0$ line bundles occur as subquotients of $h_* \calO_{C'}$ for each finite \'etale $h:C' \to C$. As there are only countably many possibilities for such $h$ (by the finite generation of $\pi_{1,\et}(C)$, proven by lifting to characteristic $0$, for example), there are only countably many possibilities for degree $0$ line bundles that occur as subquotients of $h_* \calO_{C'}$ for any finite \'etale $h:C' \to C$. Thus, for any smooth projective curve $C$, a very general degree $0$ line bundle $\calM$ will satisfy the second property, and hence both properties, above.

We now show that for a pair $(C,\calM)$ satisfying the two properties above, the group $H^1(C,\calM)$ cannot be killed by finite covers. If not, by normalising if necessary, we have a finite flat map $f:C' \to C$ such that $f^*:H^1(C,\calM) \to H^1(C',f^*\calM)$ has a non-zero kernel.  Furthermore, by replacing $C'$ with a cover if necessary, we may even assume that the extension of function fields induced by $f$ is normal. By taking invariants at the level of function fields, we can factor $f$ as $C' \stackrel{a}{\to} C'' \stackrel{b}{\to} C$ with $a$ generically \'etale and $b$ purely inseparable. Replacing $C''$ by a dominating purely inseparable map and $C'$ by the normalised base change, we may even assume that $C'' = C$ and $b = \Frob^e$ is a power of Frobenius. Our assumptions on $\calM$ imply that $H^1(C,\calM) \to H^1(C'',b^*\calM)$ is injective, and hence $a^*:H^1(C'',b^*\calM) \to H^1(C',f^*\calM)$ must have a kernel. Now consider the exact sequence
\[ 0 \to \calO_{C''} \to a_* \calO_C \to \calQ \to 0\]
where $\calQ$ is defined to be the quotient. Tensoring with $b^* \calM$ and taking cohomology, we see that if $H^1(C'',b^* \calM) \to H^1(C',f^*\calM)$ has a kernel, then $H^0(b^* \calM \otimes \calQ) \neq 0$, or equivalently that $\calM^{-p^e} = b^* \calM^{-1}$ occurs as a subsheaf of $\calQ$. We will show that this contradicts the second property above.

Since $a$ is generically \'etale, a theorem of Lazarsfeld \cite[Appendix, Proposition A]{PSL} implies that $\calQ^\vee$ is a nef vector bundle; the result in {\em loc. cit.} is stated in characteristic $0$, but this assumption is only used to ensure generic separability. The bundle $a_* \calO_C$ therefore has non-positive degree (as it is an extension of the antinef vector bundle $\calQ$ by $\calO_{C''}$). This implies that the maximal slope occuring in the Harder-Narasimhan filtration for $a_*\calO_{C'}$ is $0$. Since $\calO_C$ is a subbundle of $f_*\calO_C$ of maximal degreee, we have an induced exact sequence of semistable degree $0$ vector bundles 
\[ 0 \to \calO_{C''} \to \Fil^0(a_*\calO_{C'}) \to \Fil^0(\calQ) \to 0 \]
where $\Fil^0(\calE)$ denotes the piece of the Harder-Narasimhan filtration with slope $\geq 0$. We showed above that $\calM^{-p^e}$ occurs as a subsheaf of $\calQ$. Since $\calM^{-p^e}$ has slope $0$, it also occurs as a subsheaf of $\Fil^0(\calQ)$. Hence, $\calM^{-p^e}$ occurs as a subquotient of $\Fil^0(a_*\calO_{C'})$. On the other hand, the algebra structure on $a_* \calO_{C'}$ descends to give an algebra structure on $\Fil^0(a_*\calO_{C'})$, and thus the latter corresponds to the structure sheaf of some finite cover $h:D \to C''$ factoring the map $a$. Since $a$ is generically \'etale, the same is true for $h$. By construction, we also have $\deg(h_*\calO_D) = 0$. It now follows from Reimann-Hurwitz and Riemann-Roch that $h$ is finite \'etale. Thus, the line bundle $\calM^{-p^e}$ occurs as a subquotient of $h_*\calO_D$ for $h:D \to C$ finite \'etale, which contradicts the assumptions on the pair $(C,\calM)$, finishing the proof.

\end{example}

\begin{remark}
Example \ref{nefnotenoughfortop} requires one to work with very general line bundles, and hence does not answer the following question: does the conclusion of Proposition \ref{smithconjpos} hold for nef line bundles provided the base field is $\overline{\F_p}$?  We do not know the answer to this. Let us simply point out the reason the strategy used in Example \ref{nefnotenoughfortop} cannot work: any degree $0$ line bundle $\calM$ on a curve $C$ over $\overline{\F_p}$ is torsion since $\Pic^0(C)$ is torsion, and hence $H^1(C,\calM)$ can be killed by Proposition \ref{smithconjpos}. A natural class of examples to consider would be the surfaces and threefolds of \cite{TotaroNefSaEx}.
\end{remark}

Using Example \ref{nefnotenoughfortop}, we can easily produce an example of a nef line bundle $\calL$ on a surface $X$ whose middle cohomology cannot be killed by passage to finite covers. In fact, the bundle constructed has degree $0$ and can thus be viewed as the inverse of a nef bundle as well; this dual perspective negatively answers Question \ref{smithconj} for the case of positive or negative twists when ``weakly positive'' is taken to mean nef.

\begin{example}
\label{nefnotenough}
Let $(C,\calM)$ be as in Example \ref{nefnotenoughfortop}. Then $\calL = \calM \boxtimes \calO_C = \pr_1^*\calM$ is a nef line bundle on $X = C \times C$ with $\pr_1^*:H^1(C,\calM) \stackrel{\simeq}{\to} H^1(X,\pr_1^*\calM) = H^1(X,\calL)$. We claim that there does not exist a finite surjective morphism $\pi:Y \to X$ inducing the $0$ map $\pi^*:H^1(X,\calL) \to H^1(Y,\pi^*\calL)$. If $\pi$ was such a map, then choosing a multisection of $\pr_1 \circ \pi$ and normalising it gives a finite flat morphism $f:C' \to C$ inducing the $0$ map on $H^1(C,\calM)$. However, as shown in Example \ref{nefnotenoughfortop}, this cannot happen. 
\end{example} 

Our next example is that of a semiample line bundle $\calL$ on a surface $X$ such that the middle cohomology of $\calL^{-1}$ cannot be killed by finite covers. Thus, it negatively answers Question \ref{smithconj} for the case of negative twists when ``weakly positive'' is taken to mean even semiample, not just nef.

\begin{example}
\label{sanotenough}
Consider the bundle $\calL = \calO(2) \boxtimes \calO = \pr_1^*\calO(2)$ on $X = \P^1 \times \P^1$ over a field $k$. This is a semiample bundle with $H^1(X,\calL^{-1}) = H^1(\P^1,\calO(-2)) \otimes H^0(\P^1,\calO_{\P^1}) = k$. We claim that there is no finite surjective morphism $g:Y \to X$ inducing the $0$ map $H^1(X,\calL^{-1}) \to H^1(Y,\pi^*\calL^{-1})$. If there were such a map $g$,  then $\pr_1 \circ g:Y \to \P^1$ is an alteration inducing the $0$ map on $H^1(\P^1,\calO(-2))$. Choosing a multisection of $\pr_1 \circ g$ and normalising it gives a finite flat morphism $f:C \to \P^1$ inducing the $0$ map on $H^1(\P^1,\calO(-2))$. However, this cannot happen: the morphism of exact sequences 
\[ \xymatrix{ 0 \ar[r] & \calO(-2) \ar[r] \ar[d] & \calO \ar[r] \ar[d] & \calO_0 \oplus \calO_\infty \ar[r] \ar[d] &  0 \\
			0 \ar[r] & f_*f^*\calO(-2) \ar[r] & f_*\calO \ar[r] & f_*\calO_{f^{-1}(0)} \oplus f_*\calO_{f^{-1}(\infty)} \ar[r] & 0 }\]
gives us a morphism of exact sequences
\[ \xymatrix{  k = H^0(\P^1,\calO_{\P^1}) \ar@{^{(}->}[r] \ar[d]^{\simeq} & H^0(\P^1,\calO_0 \oplus \calO_\infty) \ar@{->>}[r]^a \ar[d]^b & H^1(\P^1,\calO(-2)) = k\ar[d]^d\\
			 k = H^0(C,\calO_C) \ar@{^{(}->}[r]^-c & H^0(C,\calO_{f^{-1}(0)} \oplus \calO_{f^{-1}(\infty)}) \ar[r] & H^1(C,f^*\calO(-2)) }\]
The surjectivity of $a$ gives that $\dim(H^0(\P^1,\calO_0 \oplus \calO_\infty)) = 2$, while the injectivity of $b$ ensures that $\dim(\im(b)) = 2$. As $\dim(\im(c)) = 1$, it follows that $\im(b)$ strictly contains $\im(c)$, and therefore, $\dim(\im(d)) = 1$ which is what we wanted.
\end{example}

Finally, we conclude by giving an example showing that the conclusion of Proposition \ref{smithconjpos} fails for nef and big line bundles.

\begin{example}
\label{nefbignotenough}
Let $(C,\calM)$ be as in Example \ref{nefnotenoughfortop}, and let $\calL$ be an ample line bundle on $C$. Let $\calE = \calL \oplus \calM$, let $X = \P(\calE)$, and let $\pi:X \to C$ be the natural projection. With $\calO_\pi(1)$ denoting the Serre line bundle on $X$, we will show the following:

\begin{itemize}
\item The line bundle $\calO_\pi(1)$ is nef.
\item The line bundle $\calO_\pi(1)$ is big.
\item The group $H^1(X,\calO_\pi(1))$ is non-zero, and cannot be annihilated by finite covers of $X$.
\end{itemize}

We will first verify that $\calO_\pi(1)$ is nef. Using the Barton-Kleiman criterion (see \cite[Proposition 6.1.18]{LazPos2}), it suffices to show that for any quotient $\calE \twoheadrightarrow \calN$ with $\calN$ invertible, we must have $\deg(\calN) \geq 0$. This claim follows from the formula
\[\Hom(\calE,\calN) = \Hom(\calL,\calN) \oplus \Hom(\calM,\calN) \]
and the fact that neither line bundle $\calL$ nor $\calM$ admits a map to a line bundle with negative degree.  

We now verify bigness of $\calO_\pi(1)$. By definition, this amounts to showing that $h^0(X,\calO_\pi(n))$ grows quadratically in $n$ (we follow the usual convention that $h^0(X,\calF) = \dim(H^0(X,\calF))$ for a coherent sheaf $\calF$ on $X$).  Standard calculations about projective space bundles show that 
\[\pi_*\calO_\pi(n) \simeq \R \pi_*\calO_\pi(n) \simeq \Sym^n(\calE) \]
for $n > 0$. The Leray spectral sequence for $\pi$ then gives us that 
\[ H^0(X,\calO_\pi(n)) = H^0(C,\Sym^n(E)) = H^0(C,\oplus_{i + j = n} \calL^i \otimes \calM^j). \]
Since $\calL$ is ample and $\calM$ has degree $0$, the Riemann-Roch estimate tells us that $H^0(C,\calL^i \otimes \calM^j)$ grows like $i$ (for big enough $i$). Hence, we find
\[ h^0(X,\calO_\pi(n)) = \sum_{i + j = n} h^0(C,\calL^i \otimes \calM^j) \sim 1 + 2 + \cdots + n = \frac{n(n-1)}{2}, \]
thereby verifying the bigness of $\calO(1)$.

It remains to check that cohomology-annihilation claim. The Leray spectral sequence shows that 
\[H^1(X,\calO_\pi(1)) = H^1(C,\calE) = H^1(C,\calL) \oplus H^1(C,\calM).\] 
In particular, this group is non-zero since the second factor is so. Moreover, the natural projection $\calE \twoheadrightarrow \calM$ defines a section $s:C \to X$ of $\pi$ such that $s^*\calO_\pi(1) \simeq \calM$. Hence, we find that $s^*$ induces a map $H^1(X,\calO_\pi(1)) \to H^1(C,\calM)$ which is simply the projection on the second factor under the preceding isomorphism. In particular, if there was a finite cover $\pi:Y \to X$ such that $\pi^*(H^1(X,\calO_\pi(1)) = 0$, then restricting $Y$ to $s:C \to X$, we would obtain a finite cover of $C$ annihilating $H^1(C,\calM)$, contradicting what we proved in Example \ref{nefnotenoughfortop}.
\end{example}

\begin{question}
The examples given above do not answer the following question: given a smooth projective variety $X$ and a nef and big line bundle $\calL$, can the group $H^i(X,\calL^{-1})$ be killed by finite covers of $X$ for $0 < i < \dim(X)$? If $\calL$ is assumed to be semiample, then Proposition \ref{smithconjneg} provides a positive answer. If the bigness condition is dropped, then Example \ref{nefnotenough} provides a negative answer. A positive answer to this question would give an ``up to finite covers'' analogue of Kawamata-Viehweg vanishing, and be quite useful for applications.
\end{question}

\section{Application: Some more global examples of D-splinters}
\label{sec:poscharglobalex}

The goal of this section is to show that the complete flag variety for $\GL_n$ is a D-splinter. Our proof relies on the results of \S \ref{sec:poscharcommentary}, and is the first non-toric example of a D-splinter in this paper (see Example \ref{toricex} for the toric case). We first record an elementary criterion to test when a finite morphism is ``split.''

\begin{lemma}
\label{critsplinter}
Let $X$ be a Gorenstein projective scheme of equidimension $n$ over a field $k$, and let $\pi:Y \to X$ be a proper morphism. Then the existence of a section of $\calO_X \to \R\pi_*\calO_Y$ is equivalent to the injectivity of $H^n(X,\omega_X) \to H^n(Y,\pi^*\omega_X)$
\end{lemma}
\begin{proof}
By the projection formula and the flatness of $\omega_X$, we have $H^n(Y,\pi^*\omega_X) = H^n(X,\omega_X \otimes \R\pi_*\calO_Y)$. Thus, the injectivity of $H^n(X,\omega_X) \to H^n(Y,\pi^*\omega_X)$ is equivalent to the injectivity of 
\[H^n(X,\omega_X) \to H^n(X,\omega_X \otimes \R \pi_*\calO_Y).\]
This map is the map on $H^n$ induced by the natural map $\omega_X \to \omega_X \otimes \R \pi_*\calO_Y$. Serre duality tells us that this injectivity is equivalent to the surjectivity of
\[ \Hom(\R\pi_*\calO_Y \otimes \omega_X,\omega_X) \to \Hom(\omega_X,\omega_X). \]
Since $\omega_X$ is invertible, the preceding surjectivity is equivalent to the surjectivity of 
\[ (\pi^*)^\vee = \ev_1:\Hom(\R\pi_*\calO_Y,\calO_X) \to \Hom(\calO_X,\calO_X) \]
induced by the natural map $\calO_X \to \R \pi_*\calO_Y$. On the other hand, the surjectivity of this map is also clearly equivalent to $\calO_X \to \R\pi_*\calO_Y$ admitting a section; the claim follows.
\end{proof}

We now record a criterion that allows us to propagate the splinter condition from a subvariety to the entire variety. The criterion is formulated in terms the existence of nice resolutions of dualising sheaves. 

\begin{proposition}
\label{splintersubtot}
Let $X$ be a Gorenstein projective variety of equidimension $n$ over a field $k$ of positive characteristic $p$. Let $i:Z \hookrightarrow X$ be a closed equidimensional subvariety that is itself Gorenstein, and let $c$ be the codimension $\dim(X) - \dim(Z)$. Assume that there exists a resolution of $\omega_Z$ of the following form:
\[ [ \omega_X = \calE_c \to \calE_{c-1} \to \cdots \to \calE_0 ] \simeq \omega_Z \]
where, for each $0 \leq i < c$, the sheaf $\calE_i$ is an iterated extension of inverses of semiample and big line bundles. If $Z$ is a splinter, so is $X$.
\end{proposition}
\begin{proof}
Let $\alpha \in H^{n-c}(X,\omega_Z) \simeq H^{n-c}(Z,\omega_Z)$ be a generator (under Serre duality). Let $f:Y \to X$ be an alteration. By assumption on $Z$, we know that $f^*(\alpha)$ is not zero in $H^{n-c}(Y,f^*\omega_Z)$. Given the natural map $\L f^*\omega_Z \to f^*\omega_Z$, we find that the pullback $\L f^*\alpha \in H^{n-c}(Y,\L f^*\omega_Z)$ is also non-zero. Note that this holds for {\em any} alteration $f:Y \to X$; this observation will be applied later in the proof to a different map.

Pulling back the given resolution for $\omega_Z$ to $Y$, we obtain a resolution
\[ [ f^*\omega_X = f^*\calE_c \to f^*\calE_{c-1} \to \cdots \to f^*\calE_0 ] \simeq \L f^*\omega_Z \]
The hypercohomology spectral sequence associated to the stupid filtration of this complex takes the form:
\[ E^{1,q}_p(Y \to X): H^q(Y,f^*\calE_p) \Rightarrow H^{q-p}(Y,\L f^*\omega_Z) \]
 We will trace the behaviour of the class $\L f^*\alpha \in H^{n-c}(Y,\L f^*\omega_Z)$ through the spectral sequence. The terms contributing to this group in the spectral sequence are $H^q(Y,f^*\calE_p)$ with $q - p = n - c$. Since $\dim(Y) = n$, the contributing terms $H^q(Y,f^*\calE_p)$ have $q < n$ whenever $p < c$. We will first show by applying Proposition \ref{smithconjneg} that these numerics imply that $\L f^*\alpha$ has to be non-zero in $H^n(Y,f^*\calE_0)$, and then we will explain why this is enough to prove the claim.

Since the bundles $\calE_i$ are assumed to be iterated extensions of inverses of semiample and big line bundles for $i < c$, the same is true for the pullbacks $f^*\calE_i$.  Proposition \ref{smithconjneg} then allows us to produce a finite surjective morphism $g:Y' \to Y$ such that $H^j(Y,f^*\calE_i) \to H^j(Y',g^*f^*\calE_i)$ is $0$ for $j < n$ (and $i < c$ still). Since we know that $\L (f \circ g)^* \alpha$ is non-zero by the earlier argument, it follows that the image of $\L f^*\alpha$ has to be non-zero in $H^{n}(Y,f^*\calE_0) = H^n(Y,f^*\omega_X)$ under the natural coboundary map $H^{n-c}(Y,\L f^*\omega_Z) \to H^n(Y,f^*\omega_X)$. 

Now note that we also have a analogous spectral sequence 
\[ E^{1,q}_p(X \to X): H^q(X,\calE_p) \Rightarrow H^{q-p}(X,\omega_Z) \]
and a morphism of spectral sequences $E^{1,q}_p(X \to X) \to E^{1,q}_p(Y \to X)$ by pulling back classes. This gives rise to the commutative square
\[ \xymatrix{ k \simeq H^{n-c}(X,\omega_Z) \ar[r]^{\delta_X} \ar[d]^{a} & H^n(X,\omega_X) \simeq k \ar[d]^b \\ H^{n-c}(Y,\L f^*\omega_Z) \ar[r]^{\delta_Y} & H^n(Y,f^*\omega_X)} \]
We have just verified that $\delta_Y \circ a$ is non-zero and, hence, injective. A diagram chase then implies that $\delta_X$ is injective and, hence, bijective. Another diagram chase then implies that $b$ is injective. By Proposition \ref{critsplinter}, we are done.
\end{proof}

\begin{remark}
Consider the special case of Proposition \ref{splintersubtot} where all the line bundles occuring in the $\calE_i$ are antiample. Since $X$ is Gorenstein, one may be tempted to say that the given proof of Proposition \ref{splintersubtot} goes through without using Proposition \ref{smithconjneg} as we can simply use Frobenius to kill cohomology after dualising. However, this is false: we applied Proposition \ref{smithconjneg} to finite covers $Y \to X$ rather than $X$ itself, and there is no reason we can suppose that $Y$ is Gorenstein. If we alter $Y$ to a Gorenstein (or even regular) scheme, then we lose ampleness, and are once again in a position where we need to use Proposition \ref{smithconjneg}.
\end{remark}

\begin{remark}
The assumptions in Proposition \ref{splintersubtot} are extremely strong. Consider the special case where $Z \hookrightarrow X$ is a divisor. In this case, the natural resolution (and, in fact, the only available one) to consider is:
\[ [ \omega_X \to \omega_X(Z) ] \simeq \omega_Z \]
The assumptions of Proposition \ref{splintersubtot} will be satisfied precisely when $\omega_X^{-1}(-Z)$ is semiample and big. This implies that $\omega_X^{-1}$ is also big. In particular, $X$ is birationally Fano.
\end{remark}

Proposition \ref{splintersubtot} looks slightly bizarre at first glance. However, it is a useful argument in inductive proofs. Here is a typical application:

\begin{proposition}
\label{flagvarsplit}
Let $V$ be a vector space of dimension $d$ over a field $k$, and let $\Flag(V)$ be the moduli space of complete flags  $(0 = F_0 \subset F_1 \subset \cdots F_{d-1} \subset F_d = V)$ in $V$. Then $\Flag(V)$ is a D-splinter.
\end{proposition}
\begin{proof}
We work by induction on the dimension $d$. The case $d = 0$ being trivial, we may assume that $\Flag(W)$ is a splinter for any vector space $W$ of dimension $\leq d-1$. If we let $\P(V)$ denote the projective space of {\em hyperplanes} in $V$, then there is a natural morphism $\pi:\Flag(V) \to \P(V)$ given by sending a complete flag $(0 = F_0 \subset F_1 \subset \cdots F_{d-1} \subset F_d = V)$ to the hyperplane $(F_{d-1} \subset V)$.  The morphism $\pi$ can easily be checked to be projective and smooth. Let $W \subset V$ be a fixed hyperplane, and let $b \in \P(V)(k)$ be the corresponding point. The fibre $\pi^{-1}(b)$ is identified with $\Flag(W)$. We will apply Proposition \ref{splintersubtot} with $Z = \Flag(W)$ and $X = \Flag(V)$ to get the desired result.

The structure sheaf $\kappa(b)$ of the point $b:\Spec(k) \hookrightarrow \P(V)$ can be realised as the zero locus of a section of $\calO(1)^{\oplus (d-1)}$ by thinking of $b$ as the intersection of $(d-1)$ hyperplanes in general position. This gives us a Koszul resolution
\[ [\calO(-(d-1)) \simeq \wedge^{d-1}(\calO(-1)^{\oplus (d-1)}) \to \dots \to \calO(-1)^{\oplus (d-1)} \to \calO] \simeq \kappa(b). \]
Twisting by $\calO(-1)$, we find a resolution
\[ [\omega_{\P(V)} \to \calM_{d-2} \to \dots \to \calM_1 \to \calM_0] \simeq \kappa(b) \]
with each $\calM_i$ a direct sum of inverses of ample line bundles with degrees between $1$ and $d-2$. Pulling this data back along $\pi$, we find a resolution
\[ [\pi^*\omega_{\P(V)} \to \pi^*\calM_{d-2} \to \dots \to \pi^*\calM_1 \to \pi^*\calM_0] \simeq \pi^*\kappa(b) = \calO_Z. \]
Twisting by the relative dualising sheaf $\omega_\pi$, we find
\[ [ \omega_\pi \otimes \pi^* \omega_{\P(V)} \to \omega_\pi \otimes \pi^*\calM_{d-2} \to \dots \to \omega_\pi \otimes \pi^*\calM_1 \to \omega_\pi \otimes \pi^*\calM_0] \simeq \omega_\pi|_Z.  \]
Since $\pi$ is smooth, we identify $\omega_X \simeq \omega_\pi \otimes \pi^*\omega_{\P(V)}$, and $\omega_Z \simeq \omega_\pi|_Z$. Thus, we obtain a resolution
\[ [ \omega_X \to \omega_\pi \otimes \pi^*\calM_{d-2} \to \dots \to \omega_\pi \otimes \pi^*\calM_1 \to \omega_\pi \otimes \pi^*\calM_0] \simeq \omega_Z   \]
with $\calM_i$ as above. Standard calculations with flag varieties (see Lemma \ref{flagvarcalc}) now show that the terms $\omega_\pi \otimes \pi^*\calM_i$ are direct sums of inverses of semiample and big line bundles. In particular, this resolution has the form required in Propositon \ref{splintersubtot}. Hence, we win by induction.
\end{proof}

We needed to calculate the positivity of certain natural line bundles on the flag variety in Proposition \ref{flagvarsplit}. Since we were unable to find a satisfactory reference, we carry out the calculation here.

\begin{lemma}
\label{flagvarcalc}
Let $V$ be an $n$-dimensional vector space over a field $k$, let $\pi:\Flag(V) \to \P(V)$ be the natural morphism. For all $i > 0$ and all $n$, the line bundles $\omega_{\pi} \otimes \pi^*\calO(-i)$ are inverses of semiample and big line bundles.
\end{lemma}
\begin{proof}
For $n = 2$, the map $\pi$ is an isomorphism, and the claim is obvious. Assume $n \geq 3$. Let 
\[ 0 = \calV_0 \subset \calV_1 \subset \cdots \subset \calV_n = V \otimes \calO_{\Flag(V)}\]
be the universal flag on $\Flag(V)$ with $\dim(V_i) = i$. For each $i \geq 1$, let $\calL_i = \calV_i/\calV_{i-1}$ be the associated line bundle. The tangent bundle of $\Flag(V)$ admits a filtration whose pieces are of the form 
\[\calHom(\calV_i,\calL_{i+1}) \simeq \calV_i^\vee \otimes \calL_{i+1} \] 
for $1 \leq i \leq n-1$. This filtration gives us the formula
\[ \omega_{\Flag(V)}^{-1} \simeq \otimes_{i = 1}^{n-1} ( \det(\calV_i)^{-1} \otimes \det(\calL_{i+1})^i).\]
Since each $\calV_i$ is filtered with pieces of the form $\calL_j$ for $1 \leq j \leq i$, we find
\[ \omega_{\Flag(V)}^{-1} \simeq \otimes_{i=1}^{n-1} (\calL_1^{-1} \otimes \calL_2^{-1} \otimes \cdots \otimes \calL_i^{-1} \otimes \calL_{i+1}^i). \]
Collecting terms, we find
\[ \omega_{\Flag(V)}^{-1} \simeq (\otimes_{i=1}^{n-1} \calL_i^{2i - n}) \otimes \calL_n^{n-1}. \]
The inverse $\calM_j$ of the line bundle $\omega_{\pi} \otimes \pi^*\calO(-j)$ can be written as
\[ \calM_j \simeq \omega_{\pi}^{-1} \otimes \pi^* \calO(j) \simeq \omega_{\Flag(V)}^{-1} \otimes \pi^*(\omega_{\P(V)} \otimes \calO(j)). \]
Using the formula for $\omega_{\Flag(V)}^{-1}$ we arrived at earlier, and the fact that $\pi$ is defined by the tautological quotient $\calO_{\Flag(V)} \otimes V \twoheadrightarrow \calL_n$, we can simplify the preceding formula to get
\[ \calM_j \simeq (\otimes_{i=1}^{n-1} \calL_i^{2i - n}) \otimes \calL_n^{n-1} \otimes \calL_n^{-n + j} \simeq (\otimes_{i=1}^{n-1} \calL_i^{2i-n}) \otimes \calL_n^{j-1}. \]
Our goal is to show that $\calM_j$ is semiample and big for $j > 0$. Being the pullback of a very ample line bundle, the factor $\calL_n^{j-1}$ is semiample and effective for $j > 0$.  Hence, it suffices to show that 
\[ \calN := \otimes_{i=1}^{2i - n} \calL_i^{2i - n} \]
is semiample and big. Since we have assumed that $n \geq 3$, the center $c = \lfloor \frac{n-1}{2} \rfloor$ is strictly positive. We may then write
\[ \calN \simeq \otimes_{k=1}^c (\calL_{n-k} \otimes \calL_k^{-1})^{\otimes (n-2k)}. \] 
Schubert calculus (see \cite[\S 10.2, Proposition 3]{FultonYT}) tells us that the line bundles $\calL_a \otimes \calL_b^{-1}$ are ample when $a > b$. In particular, all the factors in the preceding factorisation of $\calN$ are ample. Since $c \geq 1$, this factorisation is also non-empty. It follows then that $\calN$ is an ample line bundle, as desired.
\end{proof}

\begin{remark}
Proposition \ref{flagvarsplit} can be improved slightly to say that $\omega_{\pi} \otimes \pi^*\calO(-i)$ is actually the inverse of an ample line bundle. This claim follows directly from the homogeneity of $\Flag(V)$. Indeed, let $\calL$ be a semiample and big line bundle on a projective variety $X$ that is homogeneous for a connected group $G$. Let $f:X \to \P^N$ denote the map defined by a suitably large power of $\calL$. If $\calL$ was not ample, then there would be a proper curve $C \subset X$ that is contracted by $f$. By the rigidity lemma (see \cite[Proposition 6.1]{MumfordGIT}), the same is true for any curve algebraically equivalent to $C$. However, since $X$ is homogeneous, translates of $C$ under $G$ actually cover $X$. Since $G$ is connected, all translates of $C$ are algebraically equivalent to $C$. It follows then that $\dim(\im(f)) < \dim(X)$ contradicting the bigness of $\calL$.
\end{remark}

\begin{remark}
Proposition \ref{flagvarsplit} asserts that $G/B$ is a splinter when $G = \GL(V)$, and $B \subset G$ is the standard Borel subgroup. Similar arguments to the ones given above should also work when $G$ is the orthogonal group associated to a quadratic form $(V,q)$, though we have not checked that. In fact, it seems entirely plausible that the above arguments can be made to show that $G/B$ is a splinter for any algebraic group $G$. The idea would be to first show that the Bott-Samelson variety $X$ for $G$ (see \cite{BottSamelson}) is a splinter. As $X$ admits a proper birational map $\pi:X \to G/B$ satisfying $\calO_{G/B} \simeq \R \pi_* \calO_X$, it immediately follows that $X$ is also a splinter. To show the claim for $X$, one would use that $X$ comes equipped with a natural structure an explicit iterated $\P^1$-bundle with sections $X = X_n \to X_{n-1} \to \cdots \to X_1 \simeq \P^1 \to X_0 \simeq \ast$.
\end{remark}

As a corollary, we obtain a further family of examples.

\begin{corollary}
\label{partialflagvarsplit}
Let $V$ be a finite dimensional vector space, and let $X$ be a partial Flag variety for $V$. Then $X$ is a splinter. In particular, all Grassmanians $\Gr(k,n)$ are splinters.
\end{corollary}
\begin{proof}
There is a natural morphism $\pi:\Flag(V) \to X$ given by remembering the corresponding flag. It can be checked that $\pi$ is a smooth projective morphism whose fibres are iterated fibrations of projective spaces. In particular, $\R \pi_*\calO_{\Flag(V)} \simeq \pi_* \calO_{\Flag(V)} \simeq \calO_X$. The claim for $X$ now follows from that for $\Flag(V)$ proven in Proposition \ref{flagvarsplit}.
\end{proof}

\begin{remark}
Proposition \ref{flagvarsplit} and Corollary \ref{partialflagvarsplit} imply, in particular, that a (partial) flag variety is Frobenius split. The proof presented above seems to be qualitatively different proof than the standard proofs.
\end{remark}

\bibliography{mybib}

\begin{thebibliography}{MFK94}

\bibitem[AOV08]{AOVTameStacks}
Dan Abramovich, Martin Olsson, and Angelo Vistoli.
\newblock Tame stacks in positive characteristic.
\newblock {\em Ann. Inst. Fourier (Grenoble)}, 58(4):1057--1091, 2008.

\bibitem[BBD82]{BBD}
A.~A. Be{\u\i}linson, J.~Bernstein, and P.~Deligne.
\newblock Faisceaux pervers.
\newblock In {\em Analysis and topology on singular spaces, {I} ({L}uminy,
  1981)}, volume 100 of {\em Ast\'erisque}, pages 5--171. Soc. Math. France,
  Paris, 1982.

\bibitem[Bhaa]{Bhattanngrpsch}
Bhargav Bhatt.
\newblock Annihilating the cohomology of group schemes.
\newblock Available at \url{http://www-personal.umich.edu/~bhattb/papers.html}.

\bibitem[Bhab]{Bhattmixedcharpdiv}
Bhargav Bhatt.
\newblock Making cohomology p-divisible.
\newblock Available at \url{http://www-personal.umich.edu/~bhattb/papers.html}.

\bibitem[BS55]{BottSamelson}
R.~Bott and H.~Samelson.
\newblock The cohomology ring of {$G/T$}.
\newblock {\em Proc. Nat. Acad. Sci. U. S. A.}, 41:490--493, 1955.

\bibitem[BS98]{BSLocalCoh}
M.~P. Brodmann and R.~Y. Sharp.
\newblock {\em Local cohomology: an algebraic introduction with geometric
  applications}, volume~60 of {\em Cambridge Studies in Advanced Mathematics}.
\newblock Cambridge University Press, Cambridge, 1998.

\bibitem[BST11]{BSTFSingAlt}
Manuel Blickle, Karl Schwede, and Kevin Tucker.
\newblock F-singularities via alterations.
\newblock Available at \url{http://arxiv.org/abs/1107.3807}, 2011.

\bibitem[Con07]{ConradNagata}
Brian Conrad.
\newblock Deligne's notes on {N}agata compactifications.
\newblock {\em J. Ramanujan Math. Soc.}, 22(3):205--257, 2007.

\bibitem[DI87]{DeligneIllusie}
Pierre Deligne and Luc Illusie.
\newblock Rel\`evements modulo {$p\sp 2$} et d\'ecomposition du complexe de de
  {R}ham.
\newblock {\em Invent. Math.}, 89(2):247--270, 1987.

\bibitem[DM69]{DeligneMumford}
P.~Deligne and D.~Mumford.
\newblock The irreducibility of the space of curves of given genus.
\newblock {\em Inst. Hautes \'Etudes Sci. Publ. Math.}, (36):75--109, 1969.

\bibitem[Ful97]{FultonYT}
William Fulton.
\newblock {\em Young tableaux}, volume~35 of {\em London Mathematical Society
  Student Texts}.
\newblock Cambridge University Press, Cambridge, 1997.
\newblock With applications to representation theory and geometry.

\bibitem[Gro61]{EGA3_1}
A.~Grothendieck.
\newblock \'{E}l\'ements de g\'eom\'etrie alg\'ebrique. {III}. \'{E}tude
  cohomologique des faisceaux coh\'erents. {I}.
\newblock {\em Inst. Hautes \'Etudes Sci. Publ. Math.}, (11):167, 1961.

\bibitem[Gro66]{EGA4_3}
A.~Grothendieck.
\newblock \'{E}l\'ements de g\'eom\'etrie alg\'ebrique. {IV}. \'{E}tude locale
  des sch\'emas et des morphismes de sch\'emas. {III}.
\newblock {\em Inst. Hautes \'Etudes Sci. Publ. Math.}, (28):255, 1966.

\bibitem[Gro68]{SGA2}
Alexander Grothendieck.
\newblock {\em Cohomologie locale des faisceaux coh\'erents et th\'eor\`emes de
  {L}efschetz locaux et globaux {$(SGA$} {$2)$}}.
\newblock North-Holland Publishing Co., Amsterdam, 1968.
\newblock Augment{\'e} d'un expos{\'e} par Mich{\`e}le Raynaud, S{\'e}minaire
  de G{\'e}om{\'e}trie Alg{\'e}brique du Bois-Marie, 1962, Advanced Studies in
  Pure Mathematics, Vol. 2.

\bibitem[Hei02]{HeitmannDSC3}
Raymond~C. Heitmann.
\newblock The direct summand conjecture in dimension three.
\newblock {\em Ann. of Math. (2)}, 156(2):695--712, 2002.

\bibitem[HH92]{HHBigCM}
Melvin Hochster and Craig Huneke.
\newblock Infinite integral extensions and big {C}ohen-{M}acaulay algebras.
\newblock {\em Ann. of Math. (2)}, 135(1):53--89, 1992.

\bibitem[HH94]{HHTCSplinter}
Melvin Hochster and Craig Huneke.
\newblock Tight closure of parameter ideals and splitting in module-finite
  extensions.
\newblock {\em J. Algebraic Geom.}, 3(4):599--670, 1994.

\bibitem[HL07]{GLBigCM}
Craig Huneke and Gennady Lyubeznik.
\newblock Absolute integral closure in positive characteristic.
\newblock {\em Adv. Math.}, 210(2):498--504, 2007.

\bibitem[Hoc73]{HochsterDSC}
M.~Hochster.
\newblock Contracted ideals from integral extensions of regular rings.
\newblock {\em Nagoya Math. J.}, 51:25--43, 1973.

\bibitem[Hun96]{HunekeTCBook}
Craig Huneke.
\newblock {\em Tight closure and its applications}, volume~88 of {\em CBMS
  Regional Conference Series in Mathematics}.
\newblock Published for the Conference Board of the Mathematical Sciences,
  Washington, DC, 1996.
\newblock With an appendix by Melvin Hochster.

\bibitem[Ill05]{IllusieFormalGeom}
Luc Illusie.
\newblock Grothendieck's existence theorem in formal geometry.
\newblock In {\em Fundamental algebraic geometry}, volume 123 of {\em Math.
  Surveys Monogr.}, pages 179--233. Amer. Math. Soc., Providence, RI, 2005.
\newblock With a letter (in French) of Jean-Pierre Serre.

\bibitem[Kov00]{KovacsRational}
S{\'a}ndor~J. Kov{\'a}cs.
\newblock A characterization of rational singularities.
\newblock {\em Duke Math. J.}, 102(2):187--191, 2000.

\bibitem[Lau96]{LaumonCDMV1}
G{\'e}rard Laumon.
\newblock {\em Cohomology of {D}rinfeld modular varieties. {P}art {I}},
  volume~41 of {\em Cambridge Studies in Advanced Mathematics}.
\newblock Cambridge University Press, Cambridge, 1996.
\newblock Geometry, counting of points and local harmonic analysis.

\bibitem[Laz04a]{LazPos1}
Robert Lazarsfeld.
\newblock {\em Positivity in algebraic geometry. {I}}, volume~48 of {\em
  Ergebnisse der Mathematik und ihrer Grenzgebiete. 3. Folge. A Series of
  Modern Surveys in Mathematics [Results in Mathematics and Related Areas. 3rd
  Series. A Series of Modern Surveys in Mathematics]}.
\newblock Springer-Verlag, Berlin, 2004.
\newblock Classical setting: line bundles and linear series.

\bibitem[Laz04b]{LazPos2}
Robert Lazarsfeld.
\newblock {\em Positivity in algebraic geometry. {II}}, volume~49 of {\em
  Ergebnisse der Mathematik und ihrer Grenzgebiete. 3. Folge. A Series of
  Modern Surveys in Mathematics [Results in Mathematics and Related Areas. 3rd
  Series. A Series of Modern Surveys in Mathematics]}.
\newblock Springer-Verlag, Berlin, 2004.
\newblock Positivity for vector bundles, and multiplier ideals.

\bibitem[Mat80]{MatCA}
Hideyuki Matsumura.
\newblock {\em Commutative algebra}, volume~56 of {\em Mathematics Lecture Note
  Series}.
\newblock Benjamin/Cummings Publishing Co., Inc., Reading, Mass., second
  edition, 1980.

\bibitem[MFK94]{MumfordGIT}
D.~Mumford, J.~Fogarty, and F.~Kirwan.
\newblock {\em Geometric invariant theory}, volume~34 of {\em Ergebnisse der
  Mathematik und ihrer Grenzgebiete (2) [Results in Mathematics and Related
  Areas (2)]}.
\newblock Springer-Verlag, Berlin, third edition, 1994.

\bibitem[MS05]{MillerSturmfelsCCA}
Ezra Miller and Bernd Sturmfels.
\newblock {\em Combinatorial commutative algebra}, volume 227 of {\em Graduate
  Texts in Mathematics}.
\newblock Springer-Verlag, New York, 2005.

\bibitem[Mum67]{MumfordPathologies3}
D.~Mumford.
\newblock Pathologies. {III}.
\newblock {\em Amer. J. Math.}, 89:94--104, 1967.

\bibitem[PS00]{PSL}
Thomas Peternell and Andrew~J. Sommese.
\newblock Ample vector bundles and branched coverings.
\newblock {\em Comm. Algebra}, 28(12):5573--5599, 2000.
\newblock With an appendix by Robert Lazarsfeld, Special issue in honor of
  Robin Hartshorne.

\bibitem[Sin99]{SinghQGorSplinter}
Anurag~K. Singh.
\newblock {$\bold Q$}-{G}orenstein splinter rings of characteristic {$p$} are
  {F}-regular.
\newblock {\em Math. Proc. Cambridge Philos. Soc.}, 127(2):201--205, 1999.

\bibitem[Smi94]{KStcparamideals}
K.~E. Smith.
\newblock Tight closure of parameter ideals.
\newblock {\em Invent. Math.}, 115(1):41--60, 1994.

\bibitem[Smi97a]{KSmithVanishingErr}
Karen~E. Smith.
\newblock Erratum to ``vanishing, singularities and effective bounds via prime
  characteristic local algebra'', 1997.
\newblock Available for download at
  http://www.math.lsa.umich.edu/~kesmith/santaerratum.ps.

\bibitem[Smi97b]{KSfratratsing}
Karen~E. Smith.
\newblock {$F$}-rational rings have rational singularities.
\newblock {\em Amer. J. Math.}, 119(1):159--180, 1997.

\bibitem[Smi97c]{KSmithVanishing}
Karen~E. Smith.
\newblock Vanishing, singularities and effective bounds via prime
  characteristic local algebra.
\newblock In {\em Algebraic geometry---{S}anta {C}ruz 1995}, volume~62 of {\em
  Proc. Sympos. Pure Math.}, pages 289--325. Amer. Math. Soc., Providence, RI,
  1997.

\bibitem[SS10]{SinghSannai}
Akiyoshi Sannai and Anurag Singh.
\newblock Galois extensions, plus closure, and maps on local cohomology.
\newblock 2010.
\newblock Available at \url{http://www.math.utah.edu/~singh/sannai_singh.pdf}.

\bibitem[Tot09]{TotaroNefSaEx}
Burt Totaro.
\newblock Moving codimension-one subvarieties over finite fields.
\newblock {\em Amer. J. Math.}, 131(6):1815--1833, 2009.

\end{thebibliography}

\end{document}